\documentclass[12pt]{amsart}
\usepackage{amsmath,amssymb,amsbsy,amsfonts,latexsym,amsopn,amstext,
                                               amsxtra,euscript,amscd}

\newfont{\teneufm}{eufm10}
\newfont{\seveneufm}{eufm7}
\newfont{\fiveeufm}{eufm5}
\newfam\eufmfam
                \textfont\eufmfam=\teneufm \scriptfont\eufmfam=\seveneufm
                \scriptscriptfont\eufmfam=\fiveeufm
\def\bbbc{{\mathchoice {\setbox0=\hbox{$\displaystyle\rm C$}\hbox{\hbox
to0pt{\kern0.4\wd0\vrule height0.9\ht0\hss}\box0}}
{\setbox0=\hbox{$\textstyle\rm C$}\hbox{\hbox
to0pt{\kern0.4\wd0\vrule height0.9\ht0\hss}\box0}}
{\setbox0=\hbox{$\scriptstyle\rm C$}\hbox{\hbox
to0pt{\kern0.4\wd0\vrule height0.9\ht0\hss}\box0}}
{\setbox0=\hbox{$\scriptscriptstyle\rm C$}\hbox{\hbox
to0pt{\kern0.4\wd0\vrule height0.9\ht0\hss}\box0}}}}
\def\bbbq{{\mathchoice {\setbox0=\hbox{$\displaystyle\rm
Q$}\hbox{\raise 0.15\ht0\hbox to0pt{\kern0.4\wd0\vrule
height0.8\ht0\hss}\box0}} {\setbox0=\hbox{$\textstyle\rm
Q$}\hbox{\raise 0.15\ht0\hbox to0pt{\kern0.4\wd0\vrule
height0.8\ht0\hss}\box0}} {\setbox0=\hbox{$\scriptstyle\rm
Q$}\hbox{\raise 0.15\ht0\hbox to0pt{\kern0.4\wd0\vrule
height0.7\ht0\hss}\box0}} {\setbox0=\hbox{$\scriptscriptstyle\rm
Q$}\hbox{\raise 0.15\ht0\hbox to0pt{\kern0.4\wd0\vrule
height0.7\ht0\hss}\box0}}}}
\def\bbbt{{\mathchoice {\setbox0=\hbox{$\displaystyle\rm
T$}\hbox{\hbox to0pt{\kern0.3\wd0\vrule height0.9\ht0\hss}\box0}}
{\setbox0=\hbox{$\textstyle\rm T$}\hbox{\hbox
to0pt{\kern0.3\wd0\vrule height0.9\ht0\hss}\box0}}
{\setbox0=\hbox{$\scriptstyle\rm T$}\hbox{\hbox
to0pt{\kern0.3\wd0\vrule height0.9\ht0\hss}\box0}}
{\setbox0=\hbox{$\scriptscriptstyle\rm T$}\hbox{\hbox
to0pt{\kern0.3\wd0\vrule height0.9\ht0\hss}\box0}}}}
\def\bbbs{{\mathchoice
{\setbox0=\hbox{$\displaystyle     \rm S$}\hbox{\raise0.5\ht0\hbox
to0pt{\kern0.35\wd0\vrule height0.45\ht0\hss}\hbox
to0pt{\kern0.55\wd0\vrule height0.5\ht0\hss}\box0}}
{\setbox0=\hbox{$\textstyle        \rm S$}\hbox{\raise0.5\ht0\hbox
to0pt{\kern0.35\wd0\vrule height0.45\ht0\hss}\hbox
to0pt{\kern0.55\wd0\vrule height0.5\ht0\hss}\box0}}
{\setbox0=\hbox{$\scriptstyle      \rm S$}\hbox{\raise0.5\ht0\hbox
to0pt{\kern0.35\wd0\vrule height0.45\ht0\hss}\raise0.05\ht0\hbox
to0pt{\kern0.5\wd0\vrule height0.45\ht0\hss}\box0}}
{\setbox0=\hbox{$\scriptscriptstyle\rm S$}\hbox{\raise0.5\ht0\hbox
to0pt{\kern0.4\wd0\vrule height0.45\ht0\hss}\raise0.05\ht0\hbox
to0pt{\kern0.55\wd0\vrule height0.45\ht0\hss}\box0}}}}
\def\bbbz{{\mathchoice {\hbox{$\sf\textstyle Z\kern-0.4em Z$}}
{\hbox{$\sf\textstyle Z\kern-0.4em Z$}} {\hbox{$\sf\scriptstyle
Z\kern-0.3em Z$}} {\hbox{$\sf\scriptscriptstyle Z\kern-0.2em Z$}}}}

\newtheorem{theorem}{Theorem}
\newtheorem{lemma}[theorem]{Lemma}

\newtheorem{cor}[theorem]{Corollary}

\newtheorem{rem}[theorem]{Remark}

\def\squareforqed{\hbox{\rlap{$\sqcap$}$\sqcup$}}
\def\qed{\ifmmode\squareforqed\else{\unskip\nobreak\hfil
\penalty50\hskip1em\null\nobreak\hfil\squareforqed
\parfillskip=0pt\finalhyphendemerits=0\endgraf}\fi}

\def\cA{{\mathcal A}}
\def\cB{{\mathcal B}}

\def\cE{{\mathcal E}}
\def\cF{{\mathcal F}}

\def\cH{{\mathcal H}}
\def\cI{{\mathcal I}}
\def\cJ{{\mathcal J}}

\def\cP{{\mathcal P}}
\def\cQ{{\mathcal Q}}

\def\cS{{\mathcal S}}

\def\cZ{{\mathcal Z}}

\def \sf {\mathfrak s}

\def\Pf {\mathfrak P}

\def\ordp{{\,\mathrm{ord}_p}\,}

\def\ZK{\Z_\K}

\def\Res{{\mathrm{Res}}}

\newcommand{\ignore}[1]{}

\def\vec#1{\mathbf{#1}}

\def \C{\mathbb{C}}
\def \F{\mathbb{F}}
\def \K{\mathbb{K}}

\def \Z{\mathbb{Z}}

\def \R{\mathbb{R}}
\def \Q{\mathbb{Q}}
\def \N{\mathbb{N}}
\def \Z{\mathbb{Z}}

\def\mand{\qquad\mbox{and}\qquad}

\def\\{\cr}
\def\({\left(}
\def\){\right)}
\def\fl#1{\left\lfloor#1\right\rfloor}
\def\rf#1{\left\lceil#1\right\rceil}

\begin{document}

\title[Multiplicative Congruences]{Multiplicative Congruences
with Variables from Short Intervals}

\author[J.~Bourgain]{Jean~Bourgain}
\address{Institute for Advanced Study,
Princeton, NJ 08540, USA} \email{bourgain@ias.edu}

\author[M. Z. Garaev]
{Moubariz~Z.~Garaev}
\address{Centro de Ciencias Matem\'{a}ticas, Universidad Nacional Aut\'onoma de M\'{e}xico,
C.P. 58089, Morelia, Michoac\'{a}n, M\'{e}xico}
\email{garaev@matmor.unam.mx}

\author[S.~V.~Konyagin]{Sergei V.~Konyagin}
\address{Steklov Mathematical Institute,
8, Gubkin Street, Moscow, 119991, Russia} \email{konyagin@mi.ras.ru}

\author[I.~E.~Shparlinski]{Igor E.~Shparlinski}
\address{Department of Computing, Macquarie University,
Sydney, NSW 2109, Australia} \email{igor.shparlinski@mq.edu.au}

\begin{abstract}
Recently, several bounds have been obtained on the number of solutions to congruences of
the type
$$
(x_1+s)\ldots(x_{\nu}+s)\equiv  (y_1+s)\ldots(y_{\nu}+s)\not\equiv0
\pmod p
$$
modulo a prime $p$ with variables from some short intervals.
Here, for almost all $p$ and all $s$ and
also for a fixed $p$ and almost all $s$, we derive stronger bounds.
We also use similar ideas to show that for almost all primes, one
can always find an element of a large order in any rather short interval.
\end{abstract}

%

\maketitle

\section{Introduction}

For a prime $p$, let $\F_p$ be the field of residues modulo $p$.
Also, denote $\F_p^*=\F_p\setminus\{0\}$. For integers $h\ge3$ and $\nu
\ge 1$ and elements $s \in \F_p$ and $\lambda  \in \F_p^*$, we
denote by $K_{\nu}(p,h,s)$ the number of solutions of the
congruence
\begin{equation}
\label{eq:cong x,y}
\begin{split}
(x_1+s)\ldots(x_{\nu}+s)\equiv & (y_1+s)\ldots(y_{\nu}+s)\not\equiv0 \pmod p,\\
1\le x_1,\ldots&,x_{\nu}, y_1,\ldots,y_{\nu}\le h.
\end{split}
\end{equation}

Recently, a series of bounds on $K_{\nu}(p,h,s)$ as well as on the
number of solutions of a one-sided congruence
\begin{equation}
\label{eq:asym cong x,y}
\begin{split}
(x_1+s)\ldots(x_{\nu}+s)\equiv & \lambda \pmod p,\\
1\le x_1,\ldots,x_{\nu}&\le h.
\end{split}
\end{equation}
have been obtained, see~\cite{BGKS1,BGKS2,CillGar} and references therein.

In particular, it is shown in~\cite{BGKS2} that
for any fixed integer $\nu \ge 3$ we have
$$
K_{\nu}(p,h,s)\le  \(\frac{h^\nu}
{p^{\nu/e_\nu}} + 1\)
h^\nu\exp\(O\(\frac{\log h}{\log\log h}\)\),
$$
where
$$
e_\nu=\max\{\nu^2-2\nu-2,\nu^2-3\nu+4\}.
$$

Here we use and develop further some ideas of~\cite{BGKS1,BGKS2} and
obtain stronger bounds on $K_{\nu}(p,h,s)$
\begin{itemize}
\item for almost all $p$ and all $s$;
\item for a fixed $p$ and almost all  $s$.
\end{itemize}

For this purpose, we also consider the following equation
with complex $\sigma \in\C$:
\begin{equation}
\label{eq:eq x,y}
\begin{split}
(x_1+\sigma)\ldots(x_{\nu}+\sigma)= &(y_1+\sigma)\ldots(y_{\nu}+\sigma)\neq0,\\
1\le x_1,\ldots,&x_{\nu}, y_1,\ldots,y_{\nu}\le h.
\end{split}
\end{equation}
which is an analogue of the congruence~\eqref{eq:cong x,y}.

We denote by $K_{\nu}(h,\sigma)$ the number of solutions of~\eqref{eq:eq x,y}.
Here we give an asymptotic formula for $K_{\nu}(h,\sigma)$
that holds for almost all rational $\sigma$ and all irrational $\sigma$, which could be of
independent interest.

Finally, in Section~\ref{sec:Ord} we give applications of
our bounds of $K_{\nu}(p,h,s)$, and underlying ideas, to the existence of
elements of large order in short intervals.

We recall that the
notations $A \ll B$, $B \gg A$ and  $A = O(B)$ are both equivalent to the
statement that the inequality $|A| \le c\,B$ holds with some
constant $c> 0$.
Throughout the paper, any implied constants in the symbols $`\ll'$, $`\gg'$ and $`O'$ may depend on the
integer parameter $\nu\ge 1$ and sometimes on some other explicitly
mentioned parameters
and are absolute otherwise.

As usual, we use $\pi(T)$ to denote the number of primes $p\le T$.

\section{Preliminaries}

\subsection{Background on geometry of numbers}

Recall that a lattice in $\R^n$ is an additive subgroup of
$\R^n$ generated by $n$ linearly independent vectors.
Take an arbitrary convex compact and symmetric with respect to $0$
body $D\subseteq\R^n$. Recall that, for a lattice $\Gamma\subseteq\R^n$
and $i=1,\ldots,n$, the $i$th successive minimum $\lambda_i(D,\Gamma)$
of the set $D$ with respect to the lattice $\Gamma$ is defined
as the minimal number $\lambda$ such that the set $\lambda D$
contains $i$ linearly independent vectors of the lattice $\Gamma$.
Obviously, $\lambda_1(D,\Gamma)\le\ldots\le\lambda_n(D,\Gamma)$.
We need the following result given in~\cite[Proposition~2.1]{BHW}
(see also~\cite[Exercise~3.5.6]{TaoVu} for a simplified
form that is still enough for our purposes).

\begin{lemma}
\label{lem:latpoints} We have,
$$    \#(D\cap\Gamma)\le \prod_{i=1}^n \(\frac{2i}{\lambda_i(D,\Gamma)} + 1\).
$$
\end{lemma}

\subsection{Resultant bound}

Let $\Res(P_1, P_2)$ denote the resultant of two polynomials $P_1$ and $P_2$.

\begin{lemma}
\label{lem:Res} Let $H\ge1$,
$\rho,\vartheta\in\R$, and let  $M,N\ge2$ be fixed integers.
 Assume also that one of the
following conditions hold:
\begin{itemize}
\item[(i)] $\rho\ge0$;
\item[(ii)]$\vartheta\ge0$;
\item[(iii)] $\rho+\vartheta\ge -1$.
\end{itemize}
Let
$P_1(Z)$ and $P_2(Z)$ be non-constant polynomials with integer coefficients
$$
P_1(Z)=\sum_{i=0}^{M-1}a_{i}Z^{M-1-i} \mand P_2(Z)=
\sum_{i=0}^{N-1}b_{i}Z^{N-1-i}
$$
such that
$$ |a_{i}|<H^{i+\rho}, \quad i =0, \ldots, M-1,$$
$$ |b_{i}|<H^{i+\vartheta}, \quad i =0, \ldots, N-1.$$
Then
$$
\Res(P_1, P_2)\ll H^{(M-1+\rho)(N-1+\vartheta) - \rho\vartheta},
$$
where the implicit constant in $\ll$ depends only on $M$ and $N$.
\end{lemma}

\subsection{Background on algebraic integers}

Let $\K$ be a finite extension of $\Q$ and let $\ZK$ be the ring of
integers in $\K$. We recall that the logarithmic height of an
algebraic number $\alpha$ is defined as the logarithmic height
$H(P)$ of its minimal polynomial $P$, that is, the maximum logarithm
of the largest (by absolute value) coefficient of $P$.

We  need a bound of Chang~\cite[Proposition~2.5]{Chang1} on the
divisor function in algebraic number fields.

\begin{lemma}
\label{lem:Div ANF} Let $\K$ be a finite extension of $\Q$ of degree
$d = [\K:\Q]$. For any nonzero algebraic integer $\gamma\in \Z_K$ of
logarithmic height at most $H\ge 2$, the number of   pairs
$(\gamma_1, \gamma_2)$ of  algebraic integers $\gamma_1,\gamma_2\in
\Z_K$ of logarithmic  height at most $H$ with
$\gamma=\gamma_1\gamma_2$ is at most $\exp\(O(H/\log H)\)$, where
the implied constant depends on $d$.
\end{lemma}

Very often we use Lemma~\ref{lem:Div ANF}  for $d=1$ when it
gives the classical bound on the usual divisor function.

Now recall that the Mahler measure of a nonzero polynomial
$$P(Z) = a_dZ^d +\ldots+ a_1Z + a_0
= a_d\prod_{j=1}^d(Z - \xi_j) \in\C[Z]
$$
is defined as
$$M(P)= |a_d|\prod_{j=1}^d \max\{1,|\xi_j|\},$$
see~\cite[Chapter~3, Section~3]{Mig}

We recall the following estimates, that follows immediately from a
much more general~\cite[Theorem~4.4]{Mig}:

\begin{lemma}
\label{lem:HeighMahl} For any nonzero polynomial $P$ of degree $d$
the following inequality holds
$$2^{-d}e^{H(P)} \le M(P)\le (d+1)^{1/2} e^{H(P)}.$$
\end{lemma}

\begin{cor}
\label{cor:twoPol} For any nonzero polynomials $Q_1,Q_2\in\C[Z]$ we
have
$$
H(Q_1Q_2)= H(Q_1) + H(Q_2) +O(d),
$$
where $d=\deg(Q_1Q_2)$.
\end{cor}

We  also use the following  result from~\cite{BGKS2}.

\begin{lemma}
\label{lem:PolCoef} For any positive integer $\nu$ there is a
constant $C_\nu$ such that the following holds. If $P_1, P_2\in\Z[Z]$,
$P=P_1P_2$,
$$P(Z)=\sum_{j=0}^\nu u_j Z^{\nu-j}$$
and for some $A>0$ and $h>0$ the coefficients of the polynomial $P$
satisfy the inequalities
$$u_0\neq0,\qquad |u_j|\le Ah^j,\qquad j=0,\ldots,\nu,$$
then the polynomial $P_1$ has the form
$$P_1(Z)=\sum_{j=0}^\mu v_j Z^{\mu-j}$$
with
$$v_0\neq0,\quad|v_j|\le C_\nu Ah^j\quad(j=0,\ldots,\mu).$$
\end{lemma}

\subsection{Sieve arguments}

We start with recalling the following result of Iwaniec~\cite{Iwa}:

\begin{lemma}\label{Iwaniec} Let   $q_1,\ldots,q_r$ be $r\ge 2$
distinct primes. Then the number of consecutive integers each divisible
by at least one of  $q_1,\ldots,q_r$   is $O(r^2\log^2r)$.
\end{lemma}

\begin{cor}\label{mult_ind} There is an absolute constant $c>0$ such that
for any positive integer $m\ge2$ and $a\in\N$ there are  at least $c\sqrt m/\log m$
multiplicatively independent numbers $x \in [a, a+m)$.
\end{cor}

\begin{proof} Without loss of generality we can assume that $a>1$. By Lemma~\ref{Iwaniec},
one can take $r\gg\sqrt m/\log m$ so that for any primes $q_1,\ldots,q_r$
there is a number $x\in[a,a+m)$ not divisible by $q_1,\ldots,q_r$.
We  show the existence of $r+1$ multiplicatively independent numbers
$x \in [a, a+m)$ using  recursive
procedure. First, we take an arbitrary
integer $x_1\in[a,a+m)$ and an arbitrary prime divisor $q_1$ of $x_1$.
Assuming that for some $i=1,\ldots,r$ the numbers $x_1,\ldots,x_i$
and prime divisors $q_j$
of $x_j$ are chosen we use Lemma~\ref{Iwaniec}
to take $x_{i+1}\in[a,a+m)$ not divisible by $q_1,\ldots,q_i$ and
an arbitrary prime divisor $q_{i+1}$ of $x_{i+1}$. Let
integers $n_1,\ldots,n_{r+1}$ satisfy the equality
$$x_1^{n_1}\ldots x_{r+1}^{n_{r+1}}=1.$$
Since $x_1$ is the only number from $x_1,\ldots,x_{r+1}$ that is divisible
by $q_1$ we deduce that $n_1=0$. Similarly, as
$x_2$ is the only number from
$x_2,\ldots,x_{r+1}$ that is divisible by $q_2$ we
conclude that $n_2=0$,
and so on. Finally, we get $n_1=\ldots=n_{r+1}=0$ as required.
\end{proof}

We also recall the celebrated Brun--Titchmarsh theorem, see~\cite[Theorem~6.6]{IwKow}.

\begin{lemma} \label{BrunTitchmarsh}
For any integer $q$ and real $y$ with $y\ge 2q > 0$
the number $\pi(x,q,1)$ of primes $p\le y$ with $p\equiv1\pmod q$ does not exceed
$$
\pi(x,q,1) \le \frac{2y}{\varphi(q)\log(y/q)}\(1+O\(\frac1{\log(y/q)}\)\).
$$
\end{lemma}

\subsection{Distribution of divisors of shifted primes}

We need several results about divisors of shifted primes which follows from
the arguments of Ed{\H o}s and Murty~\cite{ErdMur}.

First of all we note that by~\cite[Theorem~2]{ErdMur} we have:

\begin{lemma}
\label{lem:ErdMur1}
For any function $\eta(z)>0$ with $\eta(z)\to 0$ as $z\to \infty$ and
for $T\to\infty$, for all but $o(\pi(T))$ primes $p\le T$ the number $p-1$ has no
divisor in $[T^{1/2-\eta(T)}, T^{1/2+\eta(T)}]$. \end{lemma}

Furthermore, it is easy to see that
the arguments used in the proof of~\cite[Lemma~1]{ErdMur}
also lead to the following result (in particular,
using the notation of~\cite{ErdMur}, one can notice
that the fact that $\nu = - \log \varepsilon(x)$ is not used in the proof).

\begin{lemma}
\label{lem:ErdMur2} There is an absolute constant $c_0$ such that
for any   $0<\alpha\le\gamma\le 1$
for all but $c_0\(\alpha \gamma^{-1} + 1/\log T\) \pi(T)$ primes $p\le T$
the product of all prime factors of $p-1$ that are smaller than
$T^{\alpha}$ is at most $T^{\gamma}$.
\end{lemma}

\begin{cor}
\label{cor:ErdMur2} For any constant $c$ and function $\eta(z)>0$ with $\eta(z)\to 0$ as $z\to \infty$ and
for $T\to\infty$,  for all but $o(\pi(T))$ primes $p\le T$ the product of all prime
factors of $p-1$ that are smaller than $T^{c\eta(T)/\log(1/\eta(T))}$ is at
most $T^{\eta(T)/2}$.
\end{cor}

We remark that~\cite[Theorem~2]{PomShp} gives a stronger
version of Lemma~\ref{lem:ErdMur2}  in some range of parameters $\alpha$ and $\gamma$.
However for our applications, it is essential  to  have
an  unrestricted range of
parameters $\alpha$ and $\gamma$, thus  Lemma~\ref{lem:ErdMur2}
is more suitable for our goal.

\begin{lemma}
\label{lem:NumDiv}
Let $0<\alpha\le\beta<1\le\lambda$. Then
for $T\to\infty$, for all but at most $(1/\lambda+o(1))\pi(T)$
primes $p\le T$ the number of prime divisors $q$ of $p-1$
satisfying $q\ge T^\alpha$ does not exceed
$$
\#\{q~\text{prime}~:~q \mid p-1, \ q\ge T^\alpha\}
\le \frac{2\lambda\log(\beta/\alpha)}{1-\beta}+\frac{1}{\beta}.
$$
\end{lemma}

\begin{proof} Denote by $N(p)$ the number of primes $q$
dividing $p-1$ with $T^{\alpha}\le q<T^\beta$. We have
$$\sum_{p\le T} N(p)=\sum_{T^{\alpha}\le q<T^\beta}\#
\{p\le T~:~p\equiv1\pmod q\}.$$
By Lemma~\ref{BrunTitchmarsh}, for $q<T^\beta$ we have
$$\pi(T,q,1)\le
(1+o(1))\frac{2T}{(1-\beta)q\log T}
$$
as $T\to\infty$. Hence,
$$\sum_{p\le T}N(p)\le(1+o(1))\frac{2T}{(1-\beta)\log T}
\sum_{T^{\alpha}\le q<T^\beta}\frac1q.$$
Next, by the Mertens formula, see~\cite[Equation~(2.15)]{IwKow},
$$\sum_{T^{\alpha}\le q<T^\beta}\frac1q=\log(\beta/\alpha)+o(1)$$
Therefore,
$$\sum_{p\le T}N(p)\le\frac{2T}{(1-\beta)\log T}\(\log(\beta/\alpha)+o(1)\).$$
Denoting
$$\cP=\left\{p\le T~:~N(p)>\frac{2\lambda\log(\beta/\alpha)}{1-\beta}\right\},$$
we get $\#\cP\le (1/\lambda+o(1))T/\log T$. Observing that for
a prime $p\le T$ the number of prime divisors $q\ge p^\beta$
is at most $1/\beta$, we complete the proof.
\end{proof}

\begin{lemma}
\label{lem:squaredivides}
Let $\alpha>0$. Then for $T>0$,
for all but $O\(T^{1-\alpha}\)$
primes $p\le T$ the number $p-1$ has no divisors of the
form $q^2$ with an integer  $q\ge T^\alpha$.
\end{lemma}

\begin{proof} The number of primes $p\le T$ satisfying $p\equiv1\pmod {q^2}$
does not exceed $T/q^2$. Using the inequality
$$\sum_{q\ge T^\alpha}\frac T{q^2}\le 2T^{1-\alpha},$$
the result follows.
\end{proof}

\subsection{Additive relations in multiplicative groups}

Let $\Gamma \subseteq \C^*$ be a multiplicative group of rank $r$ and
let $a_1,\ldots,a_n\in\Z$. We consider the equation
\begin{equation}
\label{eqn:EvSchSch}
a_1x_1+\ldots+a_nx_n=1,\qquad x_1,\ldots,x_n\in\Gamma.
\end{equation}

Recall that a solution of ~\eqref{eqn:EvSchSch} is {\it non-degenerate\/} provided
no subsum equals zero:
$$
\sum_{j\in \cI} a_jx_j \ne 0 \quad \text{for}\ \cI\subseteq\{1,\ldots,n\}.
$$

We  use the following result of Evertse-Schlickewei-Schmidt~\cite{EvSchSch}.

\begin{lemma}
\label{lem:EvSchSch}
The number of non-degenerate solutions of~\eqref{eqn:EvSchSch} is at most $\exp(c(n)r)$.
\end{lemma}

Recall that a solution of ~\eqref{eqn:EvSchSch} is non-degenerate provided
no subsum equals zero:
$$
\sum_{j\in \cI} a_jx_j \ne 0 \quad \text{for}\ \cI\subseteq\{1,\ldots,n\}.
$$

\begin{cor}
\label{cor:EvSchSch}
Let $\Gamma\subseteq\C^*$ be as in Lemma~\ref{lem:EvSchSch} and let $\cA\subseteq\C$ be
a finite set of cardinality $\# \cA =m$. For $a_1,\ldots,a_n\in\Z^*$ the number of solutions
of the equation
\begin{equation}
\label{eqn:corEvSchSch}
a_1x_1+\ldots+a_nx_n=0,\qquad x_1,\ldots,x_n\in\Gamma\cap \cA.
\end{equation}
is at most $O\(m^{\fl{n/2}}+\exp(O(r))m^{\fl{(n-1)/2}}\)$, where the implied constants
may depend on $n$.
\end{cor}

\begin{proof} In what follows, the implied constants may depend on $n$.

We prove the statement by induction on $n$. For $n=1,2$ the statement is trivial. We assume that for some $k\ge 3$ the statement holds for all  $n<k$, and we now prove it for $n=k$.
We can fix $x_n=b\in \Gamma\cap \cA$ such that the number $J$ of solutions of~\eqref{eqn:corEvSchSch} is not greater than $m$ times the number of solutions of the equation
\begin{equation}
\label{eqn:corEvSchSchGamma1A1}
a_1y_1+\ldots+a_{n-1}y_{n-1}=1,\qquad y_1,\ldots,y_{n-1}\in\Gamma_0\cap \cA_0,
\end{equation}
where $\Gamma_0=\langle G\cup \{-a_nb\}\rangle \subseteq\C^*$ (the group generated by $G\cup \{-a_nb\}$) and
$$
\cA_0=\{x(-a_nb)^{-1}:x\in \cA\}.
$$
To each solution of~\eqref{eqn:corEvSchSchGamma1A1} we attach a subset (possibly empty) $I\subseteq\{1,\ldots,n-1\}$ with the largest cardinality such that
$$
\sum_{j\in I}a_jy_j=0.
$$
If for a given solution there are several such subsets, then for this solution we attain one of these subsets.

Given a subset $\cI\subseteq \{1,\ldots,n-1\}$ (including an empty subset) we collect together those solutions of~\eqref{eqn:corEvSchSchGamma1A1} for which $\cI$ has been attained.
There exists a fixed $\cI\subseteq\{1,\ldots,n-1\}$ such that
\begin{equation}
\label{eqn:corEvSchSchJlessmJ1}
J\ll m J_0,
\end{equation}
where $J_0$ is the number of solutions of~\eqref{eqn:corEvSchSchGamma1A1} corresponding to the set $\cI$. We have
\begin{equation}
\label{eqn:corEvSchSchSubsetI}
\sum_{j\in I}a_jy_j=0,\quad y_j\in \Gamma_0\cap \cA_0
\end{equation}
and
\begin{equation}
\label{eqn:corEvSchSchNotSubsetI}
\sum_{j\in\{1,\ldots, n-1\}\setminus I}a_jy_j=1,\qquad y_j\in \Gamma_0.
\end{equation}
In particular,  $\cI$ is a proper subset of $\{1,\ldots,n-1\}$. Observe that
by the maximality of $\cI$ the solutions considered
in~\eqref{eqn:corEvSchSchNotSubsetI} are non-degenerate.  If $\# \cI=n-2$,
then~\eqref{eqn:corEvSchSchNotSubsetI} has at most one solution,
so that the quantity $J_0$ is bounded by the number of solutions of~\eqref{eqn:corEvSchSchSubsetI}.
Hence, by the induction hypothesis and by~\eqref{eqn:corEvSchSchJlessmJ1} we have
\begin{equation*}
\begin{split}
J\ll m (m^{\fl{(n-2)/2}}+&\exp(O(r))m^{\fl{(n-3)/2}})\\  &\ll m^{\fl{n/2}}+\exp(O(r)) m^{\fl{(n-1)/2}},
\end{split}
\end{equation*}
and the result follows.

Let now $\# \cI\le n-3$. By Corollary~\ref{cor:EvSchSch}, the number of non-degenerate
solutions of~\eqref{eqn:corEvSchSchNotSubsetI} is bounded by $\exp(O(r))$.
Furthermore, since $\# \cI\le n-3$, by the induction hypothesis the number of solutions
of~\eqref{eqn:corEvSchSchSubsetI} is
$O\( m^{\fl{(n-3)/2}}+\exp(O(r))m^{\fl{(n-4)/2}}\)$.
Therefore, by~\eqref{eqn:corEvSchSchJlessmJ1} we have
\begin{equation*}
\begin{split}
J<\exp(c(n)r) m^{\fl{(n-1)/2}}+&\exp(c(n)r)m^{\fl{(n-2)/2}}\\ = &\frac{\exp(c(n)r)}{m}m^{\fl{n/2}} +
\exp(c(n)r) m^{\fl{(n-1)/2}},
\end{split}
\end{equation*}
for some constant $c(n)$ that  depends only on $n$.

If $\exp(c(n)r)\ge m$, then the trivial estimate $J\le m^n$ implies
$$
J\le\exp(nc(n)r)
$$
and are done in this case. If $\exp(c(n)r)<m$, then
$$
J<m^{\fl{n/2}} +\exp(c(n)r) m^{\fl{(n-1)/2}}
$$
and the result follows.
\end{proof}

\begin{lemma}
\label{lem:MultIndSubset beta+j}
Let $\beta\in \C$ and $m\in \Z_{+}$. Consider a set
$$
\cA\subseteq\{\beta+j: \, 1\le j\le m\}\subseteq\C
$$
with $\# \cA>m^{\tau}$. Then there is a multiplicatively independent subset $\cA_0\subseteq \cA$ of size
$$
\# \cA_0>c(\tau)\log m,
$$
where $c(\tau)>0$ depends only on $\tau$.
\end{lemma}

\begin{proof}
Denote $\Gamma=\langle \cA\rangle\subseteq \C^*$ the multiplicative group generated by $\cA$. If
$\cA_0\subseteq \cA$ is a maximal
multiplicatively independent subset, clearly for each $x\in \cA$ we have $x^k\in \langle \cA_0 \rangle$ for some
positive integer  $k$.
Hence
$$
\text{rank}\, \Gamma=r\le \# \cA_0.
$$
We note that for an integer $n \ge 1$ the sums $z = x_1+ \ldots+ x_{n}$, $x_1, \ldots, x_{n} \in \cA$
take values
in a set $\cZ_n$ of cardinality $\#\cZ_n \le n m$. Let $N_n(z)$ be the number of such representations.
Then,
\begin{equation*}
\begin{split}
\# \{x_1&+ \ldots+x_n-x_{n+1}-\ldots-x_{2n} = 0,\ x_1, \ldots, x_{2n} \in \cA\} \\
 = &\sum_{z \in \cZ_n} N_n(z)^2 \ge \frac{1}{\#\cZ_n} \(\sum_{z \in \cZ_n} N_n(z)\)^2
=  \frac{1}{\#\cZ_n} \(\# \cA\)^{2n}
 \ge \frac{\(\# \cA\)^{2n}}{nm}.
\end{split}
\end{equation*}
Applying Corollary~\ref{cor:EvSchSch}, we get that
$$
\frac{\(\# \cA\)^{2n}}{m}<\exp(C(n)r)\(\# \cA\)^n.
$$
for some constant $C(n)$ that depends only on $n$.
Since $\# \cA>m^{\tau}$, we derive
$$
m^{\tau n-1}<\exp(C(n)r).
$$
Taking $n = \rf{2\tau^{-1}}$, it follows that $r>c(\tau)\log m$ for some positive constant $c(\tau)$ that depends only on $\tau$.
\end{proof}

\section{Multiplicative and Polynomial Congruences and Equations}

\subsection{More general congruences}

To estimate $K_{\nu}(p,h,s)$ we sometimes have to study a more
general congruence. For a prime $p$, an integer $\nu \ge 1$,
and  vectors
\begin{equation*}
\begin{split}\vec{h} & = (h_1, \ldots, h_{2\nu})\in \N^{2\nu}, \\
\vec{s} &= (s_1, \ldots, s_{2\nu}) \in \F_p^{2\nu}, \\
\vec{e} & = (e_1, \ldots, e_{2\nu}) \in \{-1,1\}^{2\nu},
\end{split}
\end{equation*}
we denote by
$K_{\nu}(p,\vec{e}, \vec{h},\vec{s})$ the number of solutions of the congruence
\begin{equation*}
\begin{split}
(x_1+s_1)^{e_1}\ldots & (x_{2\nu}+s_{2\nu})^{e_{2\nu}}\equiv  1 \pmod p,\\
1\le x_j \le & h_j,  \qquad j =1, \ldots, 2\nu.
\end{split}
\end{equation*}

The following result from~\cite{BGKS2} relates $K_{\nu}(p,\vec{e}, \vec{h},\vec{s})$ and
$K_{\nu}(p,h,s_j)$, $j =1, \ldots, \nu$. (The proof is almost instant if one expresses
 $K_{\nu}(p,h,\vec{s})$ via multiplicative character sums and uses the H{\"o}lder inequality.)

\begin{lemma}
\label{lem:Kss} We have
$$
K_{\nu}(p,\vec{e}, \vec{h},\vec{s})\le \prod_{j=1}^{2\nu}  K_{\nu}(p,h_j,s_j)^{1/2\nu}.
$$
\end{lemma}

\subsection{Multiplicative congruences and polynomial congruences with
small coefficients}

First we derive a certain condition on $s$ for which the congruence~\eqref{eq:cong x,y}
has many solutions with
\begin{equation}
\label{eq:x neq y} \{x_1,\ldots, x_{\nu}\}\cap \{y_1,\ldots,
y_{\nu}\}=\emptyset,
\end{equation}
In fact this has already appeared as a part of the argument in~\cite{BGKS2},
however here we present it in a self-contained form and in a much large generality
that we need here:

\begin{lemma}
\label{lem:CommonSols1} Let $h< 0.5 p^{1/\nu}$. Assume that for some integer $s$ there are at least
$N$ solutions to~\eqref{eq:cong x,y}
that satisfy the condition~\eqref{eq:x neq y}. Then there are
$Nh^{-1}\exp\(O(\log h/\log\log h)\)$ distinct
non-constant polynomials
$$
R(Z)=A_1Z^{\nu-1}+\ldots+A_{\nu-1}Z+A_{\nu} \in \Z[Z],
$$
with
$$
|A_i| \le \binom\nu{i} h^{i}, \qquad i =1, \ldots, \nu,
$$
and such that $R(s)\equiv 0\pmod p$.
 \end{lemma}

\begin{proof}We associate with any solution
$$
\vec{x} = (x_1,\ldots, x_{\nu}) \mand
\vec{y} = (y_1,\ldots, y_{\nu})
$$
of~\eqref{eq:cong x,y} that also satisfy~\eqref{eq:x neq y},
the polynomials
$$P_\vec{x}(Z)=(x_1+Z)\ldots(x_{\nu}+Z),
$$
and then set
\begin{equation}
\label{eq:Poly R} R_{\vec{x},\vec{y}}(Z)= P_\vec{x}(Z)-P_\vec{y}(Z).
\end{equation}
In particular,
since $R_{\vec{x},\vec{y}}(s)\equiv 0\pmod p$ and $h< 0.5 p^{1/\nu}$ it follows that
$R_{\vec{x},\vec{y}}(Z)$ is not a constant polynomial (for otherwise it is identical
to zero which is impossible by~\eqref{eq:x neq y}).
Clearly each polynomial  $R_{\vec{x},\vec{y}}(Z)$ satisfies all the
required condition.

By the Dirichlet pigeon-hole principle we have
at least $N/h$ solutions with the same $x_1=x_1^*$. We claim that
any polynomial $R$ induced by these solutions occurs at most
$\exp\(c_0(\nu) \log h/\log\log h\)$ times  for some
$c_0(\nu)$ depending only on $\nu$. Indeed, fix $R$ and
assume that $R = R_{\vec{x},\vec{y}}$.
Let $M=R(-x_1^*)$ and $z_i=-x_1^*+y_i$, $i=1,\ldots,\nu$.
We have
\begin{equation}
\label{eq: repr M} M=-P_\vec{x}(-x_1^*) -P_\vec{y}(-x_1^*)=
-P_\vec{y}(-x_1^*) =-z_1\ldots z_\nu.
\end{equation}
Using the well-known bound on the divisor function
(a special case of  Lemma~\ref{lem:Div ANF} below), we see that the number
of solutions to~\eqref{eq: repr M} is bounded by $\exp\(c_0(\nu)\log h/\log\log h\)$. Each solution determines the numbers
$y_1,\ldots,y_\nu$ and the polynomial $P_\vec{x}$, and for each $P$ there
are at most $(\nu-1)!$ solutions of~\eqref{eq:cong x,y} with $P = P_\vec{x}$. This
concludes the proof.
\end{proof}

\subsection{Common solutions to many  congruences}

We recall the following result from~\cite{BGKS2}:

\begin{lemma}
\label{lem:LinearCongr} Let   $\gamma \in  (0,1)$ and let $I$ and
$J$ be two intervals containing $h$ and $H$ consecutive integers,
respectively, and  such that
$$
h\le H<\frac{\gamma \,p}{15}.
$$
Assume that for some integer $s$ the congruence
$$
y\equiv sx\pmod p
$$
has at least $\gamma h+1$ solutions in $x\in I$, $y\in J$.
Then there exist integers $a$ and $b$ with
$$
|a| \le \frac{H}{\gamma\,h},
\qquad 0<b\le \frac{1}{\gamma},
$$
such that
$$
s\equiv a/b \pmod p.
$$
\end{lemma}

We also need the following estimate:

\begin{lemma}
\label{lem:CommonSols2} There is an absolute constant $C_0>0$ such that
if for some $s$
there are
$$
T\ge C_0 \max\{h, h^4p^{-1/2}, h^6 p^{-1}\}
$$
different triples $(U,V,W)$ with
$$
|U|\le3h,\qquad |V|\le3h^2,\qquad  |W|\le h^3,
$$
such that
$$
Us^2+Vs+W\equiv 0 \pmod p,
$$
then there are integers $r$ and $t$ satisfying conditions
$$0<|r|\ll h^{3/2}T^{-1/2}, \qquad |t|\ll h^{5/2}T^{-1/2},
\qquad rs\equiv t\pmod p.
$$
\end{lemma}

\begin{proof}
We remark that if we fix $U$ and $V$, then  there are  at most $2h^3/p+1$ possible values for $W$. In particular, we have
$$
T<36h^3(2h^3p^{-1}+1).
$$
Thus, since $C_0$ is sufficiently
large, we see that $h<c_0p^{1/3}$ for some small positive constant $c_0$.

We define the lattice
$$\Gamma = \{(u,v,w)\in\Z^3~:~us^2+vs+w\equiv 0 \pmod p\}$$
and the body
$$D = \{(u,v,w)\in\Z^3~:~|u|\le3h,\ |v|\le3h^2,\ |w|\le h^3\}.$$
We know that
$$\#(D\cap\Gamma)\ge T.$$
Therefore, by Lemma~\ref{lem:latpoints}, the successive minima
$\lambda_i=\lambda_i(D,\Gamma)$, $i=1,2,3$, satisfy the inequality
\begin{equation}
\label{ineqlambda}
\prod_{i=1}^n\min\{1,\lambda_i\}\ll T^{-1}.
\end{equation}
We can assume that $T$ is sufficiently large.
In particular, $\lambda_1\le1$. We consider separately
the following four cases.

{\it Case~1\/}: $\lambda_2>1$. Then the inequality~\eqref{ineqlambda} tells us
that $\lambda_1\ll T^{-1}$. By definition of $\lambda_1$, there is
a nonzero vector $(u,v,w)\in\lambda_1D\cap\Gamma$. We have
$$|u|\ll hT^{-1},\qquad |v|\ll h^2T^{-1} ,\qquad |w|\ll h^3T^{-1}.$$
Thus, assuming that $C_0$ is large enough, we see that $u=0$.
Since $T\gg h$, we see that $r=-v$, $t=w$ satisfy the desired bound.

{\it Case~2\/}: $\lambda_1<1/(3h)$, $\lambda_2\le 1$, and $\lambda_3>1$.
By definition of $\lambda_1$ and $\lambda_2$, there are linearly
independent vectors $(u_1,v_1,w_1)\in\lambda_1D\cap\Gamma$ and
$(u_2,v_2,w_2)\in\lambda_2D\cap\Gamma$. Moreover,
$\gcd(u_1,v_1,w_1)=1$. We see that
$|u_1|\le\lambda_1(3h)<1$. Hence,
 $u_1=0$.
 We observe that $\lambda_2\ge1/(3h)$. Indeed assume that this is not true. Then $u_2=0$ and we get that
$$
v_is+w_i\equiv 0 \pmod p,\quad i=1,2.
$$
Hence, $v_1w_2\equiv v_2w_1\pmod p$. Since the absolute values of the both hand side is not greater than
$$
3\lambda_1\lambda_2h^5<h^3/3<p/3,
$$
we obtain that $v_1w_2=v_2w_1$. This contradicts the fact that $(u_1,v_1,w_1)$ and $(u_2,v_2,w_2)$ are linearly independent.
We also note that $\lambda_1\ge1/(3h^2)$, since otherwise $u_1=v_1=0$,
implying $w_1=0$ (as $w_1\equiv 0\pmod p$ and $|w_1|\le h^3<p$).
In particular,
\begin{equation}
\label{eq:lambda1 9h3}
\lambda_1\lambda_2>1/(9h^3).
\end{equation}

By~\eqref{ineqlambda}, we have
\begin{equation}
\label{ineqlambda2}
\lambda_1\lambda_2\ll T^{-1}.
\end{equation}
We consider the polynomials
$$P_i(Z)=u_iZ^2 + v_iZ + w_i,\qquad i=1,2.$$
We see that $\deg P_1=1$ and $1\le\deg P_2\le 2$. If $\deg P_2=1$
then we conclude from Lemma~\ref{lem:Res} that
$$|\Res(P_1, P_2)|\ll h^5\lambda_1\lambda_2.$$
Using~\eqref{eq:lambda1 9h3}
we get
$$|\Res(P_1, P_2)|\ll h^8(\lambda_1\lambda_2)^2.$$
By~\eqref{ineqlambda2},
\begin{equation}
\label{ineqRes}
|\Res(P_1, P_2)|\ll h^8 T^{-2}.
\end{equation}
If $\deg P_2=2$ then we conclude from Lemma~\ref{lem:Res} that
$$|\Res(P_1, P_2)|\ll h^7\lambda_1^2\lambda_2\le h^7(\lambda_1\lambda_2)^2,$$
and due to~\eqref{ineqlambda2}, we get~\eqref{ineqRes} again.
On the other hand, we see that $\Res(P_1, P_2)$ is divisible by $p$ since
$P_1(s)\equiv P_2(s)\equiv 0\pmod p$.
If $C_0$ is chosen to be large enough,
we conclude that
$$
\Res(P_1, P_2)=0.
$$
Therefore, $\deg P_2=2$ and $P_2$ is divisible
by $P_1$ in $\Z[Z]$.
Thus,
$$
u_2w_1^2-v_2v_1w_1+w_2v_1^2=0.
$$
Hence, in view of $\gcd(v_1,w_1)=1$, we get, for some integer $k\not=0$,
\begin{equation}
\label{eq:Case2Geometry}
u_2=v_1k,\quad kw_1^2-v_2w_1+w_2v_1=0.
\end{equation}
We recall that
$$|u_2|\ll\lambda_2h, \qquad |v_2|\ll\lambda_2h^2, \qquad
|w_2|\ll\lambda_2h^3.
$$
In particular, from the first equality of~\eqref{eq:Case2Geometry} we get
$$
|v_1|\le |u_2|\ll \lambda_2 h.
$$
Together with $|v_1|\ll \lambda_1h^2$, we get that
$$
|v_1|\le (\lambda_1\lambda_2 h^3)^{1/2}.
$$
Now, the second equality of~\eqref{eq:Case2Geometry} implies that
$$
|w_1|\le 2|v_2|+2|w_2v_1|^{1/2}\ll \lambda_2h^2+(\lambda_1\lambda_2h^5)^{1/2}.
$$
Combining this with $|w_1|\ll \lambda_1h^3$ we obtain that
$$
|w_1|\ll (\lambda_1\lambda_2h^5)^{1/2}.
$$
Consequently, by~\eqref{ineqlambda2}
$$|v_1|\ll(\lambda_1\lambda_2h^3)^{1/2} \ll
h^{3/2}T^{-1/2},\qquad
|w_1|\ll(\lambda_1\lambda_2h^5)^{1/2}\ll
h^{5/2}T^{-1/2},$$
and we obtain the required inequality for $r=-v_1$,
$t=w_1$.

{\it Case~3\/}: $\lambda_1\ge1/(3h)$, $\lambda_2\le 1$, and $\lambda_3>1$.
Note that~\eqref{ineqlambda2} still holds.
Hence, $\lambda_1\le\lambda_2\ll hT^{-1}$.
By definition of $\lambda_1$ and $\lambda_2$, there are linearly
independent vectors $(u_1,v_1,w_1)\in\lambda_1D\cap\Gamma$ and
$(u_2,v_2,w_2)\in\lambda_2D\cap\Gamma$. We have
$$
|u_i|\ll h^2T^{-1},\quad |v_i|\ll h^3T^{-1},\quad |w_i|\ll h^4T^{-1}, \qquad i=1,2.
$$
As in {\it Case~2\/}, we consider polynomials
$$
P_i(Z)=u_iZ^2 + v_iZ + w_i, \qquad i=1,2,
$$
and prove that $\Res(P_1, P_2)=0$ and thus $P_1$ and $P_2$ have the same linear
factor $rZ+t$ with $\gcd(r,t)=1$. In particular,
$$
rs+t\equiv 0 \pmod p.
$$
Next, for some $i\in \{1,2\}$ we have $u_i\not=0$. For this $i$, the equality
$$
u_it^2-v_itr+w_ir^2=0,
$$
implies that
$$
u_i=rk,
$$
for some integer $k\not=0$.
In particular
$$
|r|\le |u_i|\le \lambda_i h\ll h^2/T,\quad |k|\le |u_i|\le \lambda_i h.
$$
Furthermore,
$$
kt^2-v_it+w_ir=0,
$$
implying
$$
|t|\le 2|v_i|+2(|w_i r|)^{1/2}\ll \lambda_i h^2+(\lambda_ih^3\lambda_ih)^{1/2}\ll \lambda_i h^2\ll h^3/T.
$$
This produces the required $r$ and $t$ with
$$
|r|\ll h^2T^{-1} \mand  |t|\ll h^3T^{-1}.
$$
that satisfy the desired bound (since $T\gg h$).

{\it Case~4\/}: $\lambda_3\le1$.
By definition of $\lambda_i$, there are linearly
independent vectors $(u_i,v_i,w_i)\in\lambda_iD\cap\Gamma$, $i=1,2,3$.
By~\eqref{ineqlambda}, we have $\lambda_1\lambda_2\lambda_3\ll T^{-1}$.
We consider the determinant
\begin{equation*}
D = \det  \(
  \begin{array}{cccccccc}
    u_1 & v_1 & w_1\\
    u_2 & v_2 & w_2\\
    u_3 & v_3 & w_3\\
  \end{array}
\).
\end{equation*}
Clearly,
$$
|D|\ll h^6\lambda_1\lambda_2\lambda_3 \ll h^6T^{-1}.
$$
On the other hand, from
$$u_is^2 + v_is + w_i =0,\qquad i=1,2,3,
$$
we conclude that $D$ is divisible by $p$. Therefore,
for a sufficiently large $C_0$ we derive $D=0$,
but this contradicts linear independence of the vectors
$(u_i,v_i,w_i)$, $i=1,2,3$. Thus this case is impossible.
\end{proof}

\subsection{Products with variables from intervals}
\label{sec:Equation}

First consider the case of rational values of $\sigma$ and obtain
an upper bound for the number of solutions of the equation~\eqref{eq:eq x,y}
satisfying~\eqref{eq:x neq y}. Then we consider the case
of irrational $\sigma$ and show that in this case that number is essentially
smaller that the number of trivial solutions of~\eqref{eq:eq x,y}.

\begin{lemma} \label{lem:fix_s,x}
Let $\nu \ge 1$ be a fixed integer.
Assume that $\sigma = t/r$ for some integers $r\ge 1$ and $t$
with $\gcd(r,t)=1$. Given an interval
$\cI=\{x_0+1,\ldots,x_0+u\}$, we fix some $v \in\cI$. Then the number $I$
of solutions of the equation
$$
(x_1+\sigma)\ldots(x_{\nu}+\sigma)=  (y_1+\sigma)\ldots(y_{\nu}+\sigma)\neq0,
$$
with
$$
x_1= v, \qquad
x_2,\ldots,x_{\nu}, y_1,\ldots,y_{\nu}\in\cI,
$$
and satisfying~\eqref{eq:x neq y},
does not exceed
$$
I \le \frac{u^{\nu}}{|vr+t|}
\exp\(O\(\frac{\log (u+2)}{\log\log (u+2)}\)\)
$$
\end{lemma}

\begin{proof} We rewrite the above equation as
 \begin{equation}
\label{eq2:eq x,y}
(x_1r+t)\ldots(x_\nu r+t) = (y_1r+t)\ldots(y_\nu r+t).
\end{equation}
 Given a solution to the equation~\eqref{eq2:eq x,y},  we
 fix some $j=1, \ldots, \nu$ and for $x_j$ consider
 $X_j =  |x_jr+t|$ (note that $X_j\ne 0$). Taking into account that for $i=1,\ldots,\nu$
$$y_ir+t\equiv (y_i-x_j)r \pmod {X_j},$$
we conclude from~\eqref{eq2:eq x,y}
that $(y_1-x_j)\ldots(y_\nu-x_j)r^\nu  \equiv 0 \pmod {X_j}$. Clearly $\gcd(r,X_j)=1$.  Therefore,
$(y_1-x_j)\ldots(y_\nu-x_j)  \equiv 0 \pmod {X_j}$. In particular,
$(y_1-v)\ldots(y_\nu-v)  \equiv 0 \pmod {X_j}$.

This implies that
\begin{equation}
\label{eq:Xj}
X_j =| x_jr+t | \le u^{\nu}, \qquad j = 1, \ldots, \nu.
\end{equation}

We now write $(y_1-v)\ldots(y_\nu-v) =X_1 w$ for some nonzero integer $w$
with $|w|<h^{\nu}/X_1$. Therefore, by the well-known
bound on the divisor function, for each fixed $w$ we have at most
$\exp\(O\(\log h/\log\log h\)\)$
possibilities for $y_1,\ldots,y_{\nu}$. Thus the total
number of possibilities for $y_1,\ldots,y_{\nu}$
at most
$$
\frac{h^{\nu}}{X_1} \exp\(O\(\log h/\log\log h\)\)=
\frac{h^{\nu}}{|vr+t|} \exp\(O\(\log h/\log\log h\)\) .
$$

When   $y_1,\ldots,y_{\nu}$ are fixed, using the bound~\eqref{eq:Xj}
 and  the bound  on the divisor function we obtain
 $\exp\(O\(\log h/\log\log h\)\)$
possibilities for $x_1,\ldots,x_{\nu}$, which concludes the proof.
\end{proof}

\begin{lemma} \label{lem:fix_s}
Let $\nu \ge 1$ be a fixed integer.
Assume that $\sigma = t/r$ for some integers $r\ge 1$ and $t$
with $\gcd(r,t)=1$.
Then, for $h\ge 3$,  the number $N$
of solutions of the equation~\eqref{eq:eq x,y}
satisfying~\eqref{eq:x neq y} does not exceed
$$N \le \frac{h^{\nu+1}}{\max\{hr,|t|\}}
\exp\(O\(\frac{\log h}{\log\log h}\)\).
$$
\end{lemma}

\begin{proof} We take a reordering $\{z_1,\ldots,z_h\}$
of the elements from $\{1,\ldots,h\}$ so that
\begin{equation}
\label{eq:order}
|z_1r+t|\le\ldots\le|z_hr+t|.
\end{equation}
We notice that for any $u=1,\ldots,h$ the set
$\{z_1,\ldots,z_u\}$ is a set of $u$ consecutive integers.
For $u=1,\ldots,h$ we denote by $N_u$ the number of solutions
of~\eqref{eq:eq x,y} satisfying~\eqref{eq:x neq y} such that
$\{x_1,\ldots,x_\nu,y_1,\ldots,y_\nu\}\subseteq\{z_1,\ldots,z_u\}$
and either $x_i=z_u$ or $y_i=z_u$ for some $i=1,\ldots,u$.
Clearly,
\begin{equation}
\label{eq:Nis_a_sum}
N_1=0,\quad N=\sum_{u=2}^h N_u,\quad N_u\le2\nu N^*_u\,(u\ge1)
\end{equation}
where $N^*_u$ is the number of solutions
of~\eqref{eq2:eq x,y} satisfying~\eqref{eq:x neq y} with $x_1=z_u$
and $\cI=\{z_1,\ldots,z_u\}$.
By Lemma~\ref{lem:fix_s,x}, we have
$$N_u\le\frac{u^{\nu}}{|z_ur+t|}
\exp\(O\(\frac{\log h}{\log\log h}\)\).
$$

If $|t|\ge2hr$ then $|z_ur+t|\ge|t/2|$ for any $u$, and we get
from~\eqref{eq:Nis_a_sum}
$$N\le\sum_{u=2}^h\frac{2u^{\nu}}{|t|}
\exp\(O\(\frac{\log h}{\log\log h}\)\)
\le\frac{2h^{\nu+1}}{|t|}
\exp\(O\(\frac{\log h}{\log\log h}\)\).$$
If $|t|<2hr$ then, using that there cannot be three consecutive
equal elements in the sequence~\eqref{eq:order} we obtain
the inequality
$$|z_ur+t|\ge(u-1)r/2\ge ur/4
$$
which yields
$$N\le\sum_{u=2}^h\frac{4u^{\nu-1}}{r}
\exp\(O\(\frac{\log h}{\log\log h}\)\)
\le\frac{4h^{\nu}}{r}
\exp\(O\(\frac{\log h}{\log\log h}\)\)
$$
and  completes the proof.
\end{proof}

Next, we need a bound on the number of solutions $N(h)$ to the equations
$$
uv = xy, \qquad 1 \le u,v,x,y \le h,
$$
which is given in~\cite[Theorem~3]{ACZ}:

\begin{lemma} \label{lem:ACZ}
For $h > 1$, we have
$$
N(h) = \frac{12}{\pi^2} h^2 \log h +  \kappa h^2 + O\(h^{19/13} (\log h)^{19/13}\).
$$
for some constant $\kappa$.
\end{lemma}

Note that~\cite[Theorem~3]{ACZ} gives an explicit value of $\kappa$. Furthermore,  the error term in the asymptotic formula of Lemma~\ref{lem:ACZ}
has recently been improved in~\cite{Shi}, but this has no implication on
our results.

Now we consider irrational values of $\sigma$.

We say that a solution of the equation~\eqref{eq:eq x,y} is
trivial if $(y_1,\ldots,y_\nu)$ is a permutation of
$(x_1,\ldots,x_\nu)$ and nontrivial otherwise. It is easy to see
that in the case when $\sigma$ is  transcendental  or algebraic of
degree $d\ge\nu$, the equation~\eqref{eq:eq x,y} has only trivial
solutions. Thus, it is enough to consider the case when  $\sigma$
is of  degree $d<\nu$.

\begin{lemma} \label{lem:fix_s2}
For every $\nu\ge d\ge2$, $h\ge3$ and algebraic number $\sigma$
of degree $d$, the number $N_\nu(h,\sigma)$
of solutions of the equation~\eqref{eq:eq x,y}
satisfying~\eqref{eq:x neq y} does not exceed
$$N_\nu(h,\sigma) \le h^{\nu-d+1}\exp\(O\(\frac{\log h}{\log\log h}\)\).$$
\end{lemma}

\begin{proof} Let $P(X)\in \Z[X]$ be an irreducible
primitive polynomial such that $P(\sigma)=0$. Clearly, $\deg P=d$.
Next, we consider the polynomial
$$
G(X)=(x_1+X)\ldots (x_{\nu}+X)-(y_1+X)\ldots (y_{\nu}+X)
$$
and factor $G$ into irreducible over $\Q$ polynomials $f\in\Z[X]$.
Since $G(\sigma)=0$, for some irreducible factor $f\in\Z[X]$
of $G$ we have $f(\sigma)=0$. Thus, $f(X)=rP(X)$ for some rational
number $r$, and since $P$ is primitive, $r$ is integer. Thus,
$G(X)=A(X)P(X)$ for some $A\in\Z[X]$.

We have
$$A(X)=\sum_{j=0}^{\nu-d-1}a_jX^{\nu-j}.$$
By Lemma~\ref{lem:PolCoef}, we have
\begin{equation}
\label{est_a_j}
a_j\ll h^{j+1}\quad j=0,\ldots,\nu-d-1,
\end{equation}
where the implied constants depend only on $\nu$.

For $v=1,\ldots,h$,  we now estimate the number $N_v$
of solutions of the equation~\eqref{eq:eq x,y}
satisfying~\eqref{eq:x neq y} with $x_1=v$. Clearly,
\begin{equation}
\label{N=sum_N_v}
N=\sum_{v=1}^h N_v.
\end{equation}
From $G(-v)=P(-v)A(-v)$ we get that
$$
-(y_1-v)\ldots (y_\nu-v)=P(-v)A(-v).
$$
By~\eqref{est_a_j} we have $A(-v)\ll h^{\nu-d}$. Therefore,
there are $O\(h^{\nu-d}\)$ possible values for
$|(y_1-v)\ldots (y_\nu-v)|$. In turn, this implies that
there are at most
$h^{\nu-d}\exp\(O\({\log h}/{\log\log h}\)\)$
possible values for $y_1,\ldots,y_{\nu}$. Once the variables
$y_1,\ldots,y_{\nu}$ are fixed,
by Lemma~\ref{lem:Div ANF} we see that there are
$\exp\(O\({\log h}/{\log\log h}\)\)$
possibilities for
$x_2,\ldots,x_{\nu}$ (while $x_1$ is fixed by $x_1=v$). Thus,
$$N_v\le h^{\nu-d}\exp\(O\(\frac{\log h}{\log\log h}\)\), \qquad v=1,\ldots,h,
$$
and using~\eqref{N=sum_N_v} completes the proof of the lemma.
\end{proof}

We say that a solution $ x_1,\ldots,x_{\nu}, y_1,\ldots,y_{\nu}$ to the equation~\eqref{eq:eq x,y}
is {\it trivial\/} if $(y_1,\ldots,y_\nu)$ is a permutation
of $(x_1,\ldots,x_\nu)$.

\begin{theorem}
\label{thm:IrratProdSet s} For every $\nu\ge d\ge2$, $h\ge3$ and
algebraic number $\sigma$ of degree $d$ the number
$M_\nu(h,\sigma)$ of nontrivial solutions of~\eqref{eq:eq x,y}
satisfies the inequality
$$M_\nu(h,\sigma)\le  h^{\nu-d+1}\exp\(O\(\frac{\log h}{\log\log h}\)\).$$
\end{theorem}

\begin{proof} We  use induction on $\nu\ge d$.
For $\nu=d$ all solutions are trivial, and there is nothing to prove.
We verify the assertion for $\nu>d$ assuming that it holds for
$\nu-1$. Using induction hypothesis we conclude that the number
of nontrivial solutions of~\eqref{eq:eq x,y} such that
condition~\eqref{eq:x neq y} does not hold is
bounded by
\begin{equation*}
\begin{split}
(2h\nu)\cdot h^{\nu-1-d+1} \exp\(O\(\frac{\log h}{\log\log h}\)\)\\
=  h^{\nu-d+1}& \exp\(O\(\frac{\log h}{\log\log h}\)\).
\end{split}
\end{equation*}

It suffices to add the number of solutions of~\eqref{eq:eq x,y}
satisfying~\eqref{eq:x neq y}. Using Lemma~\ref{lem:fix_s2}
we complete the proof.
\end{proof}

Since the number of trivial solutions of~\eqref{eq:eq x,y}
is $\nu! h^{\nu} + O\(h^{\nu-1}\)$, we get the following:

\begin{cor}
\label{cor:IrratProdSet s}
For every $\nu\ge1$, $h\ge3$ and irrational number $\sigma$
of degree $d$ we have the asymptotic formula
$$
K_\nu(h,\sigma) =  \nu! h^{\nu} + O\(h^{\nu-1}
\exp\(O\(\frac{\log h}{\log\log h}\)\)\),
$$
where the implicit constants depend only on $\nu$.
\end{cor}

\begin{rem} It is certainly interesting to find best possible
value of the exponent of $h$ in Theorem~\ref{thm:IrratProdSet s}.
\end{rem}

\begin{rem}
It is also interesting to understand  for which $d$ and $\nu$
there are nontrivial solutions of~\eqref{eq:eq x,y} for every
algebraic number $\sigma$ of degree $d$ and a sufficiently large
$h$.
\end{rem}

\begin{rem} One can try to get   more precise forms of Theorem~\ref{thm:IrratProdSet s}
and Lemma~\ref{lem:fix_s2}
depending on the coefficients of the minimal
polynomial $P$ for $\sigma$ (similarly to the estimates in  Lemma~\ref{lem:fix_s,x}).
Such an improvement is of independent interest, however
does not help the main goal of  this work.
\end{rem}

\section{Multiplicative Congruences and Equations for Almost all Parameters}

\subsection{Bounds on the number of solutions of multiplicative congruences for almost all $p$}

\begin{theorem}
\label{thm:GProdSet AlmostAll} Let $\nu \ge 1$ be a fixed integer.
Then for a sufficiently large positive integer  $T$, $h\ge3$,
for all but $o(T/\log^2 T)$ primes $p \le T$, if $3\le h<T$
then for any  $s \in \F_p$, we have the bound
$$
K_{\nu}(p,h,s) \le \(h^\nu + h^{2\nu-1/2}T^{-1/2}\)
\exp\(O\(\frac{\log h}{\log\log h}\)\).
$$
\end{theorem}

\begin{proof} We note that for $\nu=1$ the result is trivial and we prove
it for $\nu\ge2$ by induction on $\nu$.

Let
$$
H_\nu  = T^{1/(2\nu-1)} (\log T)^{-5/(2\nu-1)}.
$$

We consider the quadruples of polynomials $(P_1,Q_1,P_2,Q_2)$ with
\begin{equation}
\label{eq:formP,Q} P_i(Z)=(x_{1,i}+Z)\ldots(x_{\nu,i}+Z),\quad Q_i(Z)
=(y_{1,i}+Z)\ldots(y_{\nu,i}+Z),
\end{equation}
and
$$1\le x_{1,i},\ldots, x_{\nu,i}, y_{1,i},\ldots,y_{\nu,i}\le 2h,
$$
for $i=1,2$.
Denote
$$R_1=P_1-Q_1,\qquad R_2=P_2-Q_2.$$

Clearly, we have $|\Res(R_1,R_2)| = h^{O(1)}$. If $h\le H_\nu$, then there
are at most $O(H_\nu^{2\nu-1} \log H_\nu) = o(T/\log^2 T)$ primes $p$ that divide
$\Res(R_1,R_2)$ with $|\Res(R_1,R_2)|\ne 0$ for at least
$h^{2\nu+1}$ quadruples $(P_1,Q_1,P_2,Q_2)$ (since each quadruples
may correspond to at most $O(\log  H_\nu)$ distinct primes
and we have $O(h^{4\nu}$) distinct quadruples). If $h> H_\nu$, then there
are at most $o(T/\log^2 T)$ primes $p$ that divide
$\Res(R_1,R_2)$ with $|\Res(R_1,R_2)|\ne 0$ for at least
$h^{4\nu}\log^4T/T$ quadruples $(P_1,Q_1,P_2,Q_2)$. Also, by the induction
hypothesis, there are $o(T/\log^2 T)$ primes $p$ not satisfying the
condition
\begin{equation}
\label{Knu-1}
K_{\nu-1}(p,h,s) \le \(h^{\nu-1} + h^{2\nu-5/2}T^{-1/2}\)
\exp\(O\(\frac{\log h}{\log\log h}\)\).
\end{equation}

We now fix one of the remaining primes $p\le T$ and estimate, for
any $s\in\F_p$, the cardinality $N$ of the set $\cQ$ of quadruples
of polynomials $(P_1,Q_1,P_2,Q_2)$ such that for $i=1,2$ the
polynomials $P=P_i, Q=Q_i$   are of the form~\eqref{eq:formP,Q} with
$x_{1,i},\ldots, x_{\nu,i}, y_{1,i},\ldots,y_{\nu,i}$ that
satisfy~\eqref{eq:cong x,y} and~\eqref{eq:x neq y},
and, moreover, $\Res(R_1,R_2)\neq0$. However,
$\Res(R_1,R_2)\equiv0\pmod p$. For any such quadruple, for any
$t=1,\ldots,h$, and for $i=1,2$ we consider polynomials
$$P_{i,t}(Z)=P_i(Z+t),\quad Q_{i,t}(Z)=Q_i(Z+t),\quad R_{i,t}(Z)=R_i(Z+t).$$
We get $Nh$ quadruples $(P_{1,t},Q_{1,t},P_{2,t},Q_{2,t})$. They are
not necessarily distinct, but the multiplicity of any quadruple is
at most $\nu$ since $R_{i,t}(s-t)\equiv0\pmod p$ and any polynomial
$R_i$ cannot coincide with $R_{i,t}$ for more than $\nu$ distinct
$t$. Thus, for $h\le H_\nu$ we have $Nh/\nu < h^{2\nu+1}$, or
\begin{equation}
\label{eq:estN_small_h} N<\nu h^{2\nu}.
\end{equation}
For $h> H_\nu$ we have $Nh/\nu <h^{4\nu}\log^4T/T$, or
\begin{equation}
\label{eq:estN_large_h} N<\nu h^{4\nu-1}T^{-1}\log^4T .
\end{equation}
Let $M$ be the cardinality of the set $\cP$ of solutions
of~\eqref{eq:cong x,y} satisfying~\eqref{eq:x neq y}. We assume that
$M\ge \nu (h^\nu + h^{2\nu-1/2}T^{-1/2}\log^2T)$. By the Dirichlet pigeon-hole
principle, there exists
$(P_1,Q_1)\in\cP$ such that the number of pairs $(P_2,Q_2)\in\cP$
with $(P_1,Q_1,P_2,Q_2)\in\cQ$ is at most
$h^{2\nu-1/2}T^{-1/2}\log^2T$. Therefore,
the number
$M_0$ of pairs $(R_1,R_2)\in\cP$ with $\Res(R_1,R_2)=0$ satisfies
$M_0 \ge M-h^{2\nu-1/2}\log^2T\ge M/2$.
Thus,  by
Corollary~\ref{cor:twoPol},
we find an algebraic number $\beta$ of
logarithmic height $O(\log h)$  in an extension $\K$ of $\Q$ of
degree $[\K:\Q] \le \nu$ such that the equation
\begin{equation}
\label{eq:ConcentrationProdAlgebraic}
(x_1+\beta)\ldots
(x_{\nu}+\beta)= (y_1+\beta)\ldots (y_{\nu}+\beta)\neq0,
\end{equation}
where
$$1\le x_i,y_i\le h  \qquad i=1, \ldots, \nu,
$$
has at least $M_0/\nu$ solutions.
 Now we have that
$$\beta =\frac{\alpha}{q},
$$
where $\alpha$ is an algebraic integer of height at most $O(\log h)$
and $q$ is a positive integer $q\ll h^{\nu}$, see~\cite{Nar}. From
the basic properties of algebraic numbers it now follows that the
numbers
$$qx_i+\alpha \mand qy_i+\alpha, \qquad i =1, \ldots, \nu,
$$
are algebraic integers of $\K$ of height at most $O(\log h)$.

Using  Lemma~\ref{lem:Div ANF},
we conclude that for a sufficiently large $h$ the
equation~\eqref{eq:ConcentrationProdAlgebraic}
has at most
$$h^\nu \exp\(O\(\frac{\log h}{\log\log h}\)\)$$
solutions. Therefore, the same estimate holds for the number of solutions
of~\eqref{eq:cong x,y} satisfying~\eqref{eq:x neq y}. By~\eqref{Knu-1} we have a similar estimate
for the number of solutions of~\eqref{eq:cong x,y} not satisfying~\eqref{eq:x neq y}. This
completes the proof of the theorem.
\end{proof}

Taking the sum over $h=3\times2^j$, $j\ge0$, we get the same exceptional
set for all $h$.

\begin{cor}
\label{cor:GProdSet AlmostAll} Let $\nu \ge 1$ be a fixed integer.
Then for a sufficiently large positive integer  $T$,
for all but $o(\pi(T))$ primes $p \le T$, for any $3\le h<T$
and for any $s \in \F_p$ we have the bound
$$
K_{\nu}(p,h,s) \le \(h^\nu + h^{2\nu-1/2}T^{-1/2}\)
\exp\(O\(\frac{\log h}{\log\log h}\)\).
$$
\end{cor}

Clearly for $h = O(T^{1/(2\nu-1)})$ the first term in Theorem~\ref{thm:GProdSet AlmostAll}
and Corollary~\ref{cor:GProdSet AlmostAll} dominates and both bounds take an almost optimal
form
$$
K_{\nu}(p,h,s) \le h^\nu \exp\(O\(\frac{\log h}{\log\log h}\)\).
$$

For a set $\cA \subseteq \F_p$ we denote
$$
\cA^{(\nu)} = \{a_1 \ldots a_\nu~:~a_1, \ldots, a_\nu \in \cA\}.
$$

\begin{cor}
\label{cor:ProdSet Fp AlmostAll}  Let $\nu \ge 1$ be a fixed integer.
Then for a sufficiently large positive integer  $T$,
for all but $o(\pi(T)$ primes $p \le T$, if $3\le h<T$
then for any  $s \in \F_p$, for the set
$$
\cA = \left\{ x+s ~:~ 1\le x\le  h \right\} \subseteq \F_p
$$
we have
$$
\# \cA^{(\nu)} \ge\min\(h^{\nu},(hT)^{1/2}\)
\exp\(O\(\frac{\log h  }{ \log \log h }\)\).
$$
\end{cor}

\subsection{Asymptotic formula for the number of solutions of multiplicative
equations for almost all $\sigma \in \C$}

\begin{theorem} \label{thm:RatProdSet s}
For every $\nu\ge1$, $h\ge3$ and $\varepsilon$ with $2\ge \varepsilon > 0$
for all but $O(h^{1+\varepsilon})$ values of $\sigma \in \C$ we have
the asymptotic formula
$$
K_\nu(h,\sigma) =  \nu! h^{\nu} + O\(h^{\nu-\varepsilon/2}
\exp\(O\(\frac{\log h}{\log\log h}\)\)\).
$$
\end{theorem}

\begin{proof}
 By Corollary~\ref{cor:IrratProdSet s}, we have the desired estimate for
irrational $\sigma$, so it is enough to consider only rational $\sigma$.

We denote
$$\cS=\{s=t/r~:~r, t\in\Z,\ |t|\le h^{1+\varepsilon/2},\  1\le
r\le h^{\varepsilon/2}\}.$$
Clearly, $\#\cS=O(h^{1+\varepsilon})$. It suffices
to prove the desired asymptotic formula for $s\in\Q\setminus\cS$.
We  follow the proof of Theorem~\ref{thm:IrratProdSet s}.
We say that the solution of~\eqref{eq:eq x,y} is {\it trivial\/} if
$(y_1,\ldots,y_\nu)$ is a permutation of $(x_1,\ldots,x_\nu)$.
The number of trivial solutions is
$$\nu! h^{\nu} +O(h^{\nu-1}).$$
Using induction on $\nu$, we prove the number of nontrivial solutions
is
$$
O\(h^{\nu-\varepsilon/2}
\exp\(C(\nu)\frac{\log h}{\log\log h}\)\),$$
where $C(k)$ depends only on $k$.
For $\nu=1$ all solutions are trivial. We   prove the assertion
for $\nu>1$ assuming that it holds for
$\nu-1$. Using induction hypothesis we conclude that the number
of nontrivial solutions of~\eqref{eq:eq x,y} such that
condition~\eqref{eq:x neq y} does not hold is bounded by
$$
O\(h^{\nu-\varepsilon/2}
\exp\(C(\nu-1)\frac{\log h}{\log\log h}\)\).
$$
To estimate the number $N$ of solutions of~\eqref{eq:eq x,y}
satisfying~\eqref{eq:x neq y} we use Lemma~\ref{lem:fix_s}. We write
$s=t/r$ where $t\in\Z$, $r\in\N$, $\gcd(r,t)=1$.
Since $s\not\in\cS$, we have $\max\{hr,|t|\}>h^{1+\varepsilon/2}$.
Hence, by Lemma~\ref{lem:fix_s},
$$N\le h^{\nu-\varepsilon/2}
\exp\(O\(\frac{\log h}{\log\log h}\)\),$$
as required.
\end{proof}

We now note that the bound of Theorem~\ref{thm:RatProdSet s}
on the size of exceptional set of $\sigma$
is quite precise even if  one restricts $\sigma$ to integer values.

\begin{theorem}
\label{thm:RatProdSet s LB}
For every $\nu\ge1$ and
$1/2>\varepsilon > 0$
for all integers $s$ with
$1 \le s \le h (\log h)^{1/2-\varepsilon}$ we have
$$
K_\nu(h,s)  \gg h^\nu (\log h)^{2\varepsilon}.
$$
\end{theorem}

\begin{proof}
As before, let $N(h)$ be the number of integer solutions of the equation
$$
uv = xy, \qquad 1 \le u,v,x,y \le h.
$$
We note that
\begin{equation*}
\begin{split}
N(h+s) - N(s) \le 4 \#\{(u,v&,x,y)\in \Z^4~:~uv = xy,\\
&s+1\le u\le s+h, 1\le v,x,y\le s+h\}.
\end{split}
\end{equation*}
Take an arbitrary prime $p \ge 2(s+h)^2$. Then, in the above range of
variables, the equation $uv = xy$
is equivalent to the congruence $uv \equiv xy\pmod p$.
Using Lemma~\ref{lem:Kss}, we derive
\begin{equation*}
\begin{split}
\#\{(u,v,x,y) \in \Z^4~:~uv = xy,\ s+1\le u\le s+&h,  \  1\le v,x,y\le s+h\}\\
\le &
K_2(h,s)^{1/4}N(h+s)^{3/4}.
\end{split}
\end{equation*}
Thus, we have
\begin{equation}
\label{eq:upper}
N(h+s) - N(s)\le 4 K_2(h,s)^{1/4}N(h+s)^{3/4}.
\end{equation}
We now consider two cases.

{\it Case~1\/}:  $s < h/2$. Then for a sufficiently large $h$, by Lemma~\ref{lem:ACZ} have
$$
N(h+s) \ge N(h) \ge  h^2 \log h \quad \text{and}\quad
N(s) \le   N(\fl{h/2}) \le  0.5 h^2 \log h.
$$
Thus,
$$
N(h+s)-N(s)\ge 0.5h^2\log h.
$$
By~Lemma~\ref{lem:ACZ} again, we have
$$
N(h+s)\le 2h^2\log h.
$$
Inserting these bounds into~\eqref{eq:upper}, we obtain $K_2(h,s)  \gg h^2 \log h$ and the result follows in this case.

{\it Case~2\/}:   $h/2 < s \le h \log h$. Then from Lemma~\ref{lem:ACZ} we get
$$
N(h+s) - N(s)  \gg sh \log h.
$$
By Lemma~\ref{lem:ACZ} we also have
$$
N(h+s)\ll s^2\log h.
$$
Combining these bounds with  bounds~\eqref{eq:upper}, we obtain $K_2(h,s)  \gg h^2 (\log h)^{2\varepsilon}$.
\end{proof}

\subsection{Asymptotic formula for the number of solutions of multiplicative congruences
for almost all $s\in\F_p$}

We start with the cases of $\nu=2$ and $\nu=3$ where we have stronger results that
in the case of arbitrary $\nu$.

\begin{theorem}
\label{thm:GProdSet nu=2 s}
For every $\varepsilon$ with $2> \varepsilon > 0$ and
$$
3\le h \le p^{1/(2  + \varepsilon/2)}
$$
for all but $O(h^{1+\varepsilon})$ values of $s \in \F_p$ we have the asymptotic formula
$$
K_{2}(p,h,s) =  2 h^2 + O\(h^{2-\varepsilon/2}
\exp\(O\(\frac{\log h}{\log\log h}\)\)\).
$$
\end{theorem}

\begin{proof}  Assume that $s$ is such that
\begin{equation}
\label{eq:nu2}
(x_1+s) (x_2+s)\equiv  (y_1+s) (y_2+s)\not\equiv0 \pmod p
\end{equation}
is satisfied for at least $ h^{2-\varepsilon/2}\exp\(C_1\log h/\log\log h\)$
choices of integers $1\le x_1,x_2, y_1,y_2\le h$
with $\{x_1,x_2\} \ne \{y_1,y_2\}$, where  $C_1>0$ is a sufficiently
large constant.

Using Lemma~\ref{lem:CommonSols1}, we see that $s$ satisfies at least
$C_2h^{1-\varepsilon/2}$ distinct linear congruences
$As + B \equiv 0 \pmod p$ with $0< |A| \le 2h$  and $|B| \le h^2$
where  $C_2>0$ is another sufficiently large constant.
Thus by Lemma~\ref{lem:LinearCongr}, applied with $\gamma = 16h^{-\varepsilon/2}$
we obtain that $s$ can take at most $O(h^{1+\varepsilon})$ possible values.
\end{proof}

\begin{theorem}
\label{thm:GProdSet nu=3 s}
For every $\varepsilon$ with $2> \varepsilon > 0$ and
$$
3\le h \le p^{1/(4 +\varepsilon)}
$$
for all but $O(h^{1+\varepsilon})$ values of $s \in \F_p$ we have
the asymptotic formula
$$
K_{3}(p,h,s) =  6 h^3 + O\(h^{3-\varepsilon/2}
\exp\(O\(\frac{\log h}{\log\log h}\)\)\).
$$
\end{theorem}

\begin{proof}  By Theorem~\ref{thm:GProdSet nu=2 s}, we see that there is a set
$\cS\subseteq\F_p$ of cardinality $O(h^{1+\varepsilon})$ such that
for all $s\in\F_p\setminus\cS$, the equation~\eqref{eq:nu2} has at most
$h^{2-\varepsilon/2} \exp\(O\(\log h/\log\log h\)\)$
solutions with $\{x_1,x_2\}\neq\{y_1,y_2\}$.

It is now easy to  see that for $s \not \in \cS$ we have
$$
K_{3}(p,h,s) = 6 h^3 + O\(h^{3-\varepsilon/2}
\exp\(O\(\frac{\log h}{\log\log h}\)\) + N\),
$$
where
$N$ denotes the contribution  to $K_{3}(p,h,s)$  of solutions with
$x_i\not=y_j$, $1 \le i,j \le 3$.

Assume that
$$
N \ge  h^{3-\varepsilon/2}\exp\(C_1\frac{\log h}{\log\log h}\)
$$
for some appropriate constant $C_1>0$.

We see from Lemma~\ref{lem:CommonSols1} that
for sufficiently large $h$ and
for another appropriate constant $C_2>0$ there are
at least $T = \fl{C_2h^{2-\varepsilon/2}}$ different triples $(U,V,W)$ with
$$
|U|\le3h,\qquad |V|\le3h^2,\qquad  |W|\le h^3,
$$
such that
$$
Us^2+Vs+W\equiv 0 \pmod p.
$$
Clearly, $T$ satisfies the conditions of Lemma~\ref{lem:CommonSols2},
thus there are some integers $r$ and $t$ with
$0<|r| \ll h^{1/2+\varepsilon/4}$,
$t  \ll h^{3/2+\varepsilon/4}$, and $rs\equiv t\pmod p$.

Clearly we can assume that $\gcd(r,t)=1$.

We see that
$$
(x_1+s) (x_2+s) (x_3+s)\equiv  (y_1+s) (y_2+s)(y_3+s)\not\equiv0 \pmod p
$$
implies
\begin{equation}
\begin{split}
\label{eq:equiv form}
(x_1+x_2 + x_3  -y_1 &-y_2-y_3) t^2 \\
+ (x_1x_2 + x_1 &x_3 +  x_2x_3-y_1y_2- y_1y_3 - y_2y_3) rt \\
& + (x_1x_2x_3  -y_1y_2y_3)r^2 \equiv0 \pmod p.
\end{split}
\end{equation}
The absolute value of the left-hand side of~\eqref{eq:equiv form} is
at most $14 h^{4+\varepsilon/2}<p$, and we get the equation
\begin{equation*}
\begin{split}
(x_1+x_2 + x_3  -y_1  -y_2-y_3) &t^2 \\
+ (x_1x_2 + x_1  x_3 +  x_2x_3& - y_1  y_2- y_1y_3 - y_2y_3) rt \\
& + (x_1x_2x_3  -y_1y_2y_3)r^2 =0,
\end{split}
\end{equation*}
which, in turn, we transform into an   equivalent  equation
$$
(x_1+\sigma) (x_2+\sigma) (x_3+\sigma)=
(y_1+\sigma) (y_2+\sigma)(y_3+\sigma)
$$
with a rational $\sigma= t/r$. Since $\gcd(t,r)=1$
different values of $s$ lead to different values of $\sigma$.
Using  Theorem~\ref{thm:RatProdSet s} we conclude the proof.
\end{proof}

Furthermore, we have the following general
form of Theorem~\ref{thm:GProdSet nu=3 s} which holds for  an arbitrary $\nu \ge 1$.

\begin{theorem}
\label{thm:GProdSet nu>3 s}
For every $\nu\ge2$, there exists some positive $\gamma(\nu)$
such that for
$$
3\le h \le \gamma(\nu) p^{1/(\nu^2-1)}
$$
and $0<\varepsilon\le2$
for all but $O(h^{1+\varepsilon})$ values of $s \in \F_p$ we have
the asymptotic formula
$$
K_{\nu}(p,h,s) =  \nu! h^\nu + O\(h^{\nu-\varepsilon/2}
\exp\(O\(\frac{\log h}{\log\log h}\)\)\).
$$
\end{theorem}

\begin{proof} Without loss of generality we may assume that $p$
is large enough. As before, we say that a solution
of~\eqref{eq:asym cong x,y} is {\it trivial\/} if $(y_1,\ldots,y_\nu)$ is a permutation
of $(x_1,\ldots,x_\nu)$. Let $S$ be the set of all $s\in\F_p$ such that
~\eqref{eq:asym cong x,y} has at least one nontrivial solution. If
$s\in\F_p\setminus\{S\}$ then
$$K_{\nu}(p,h,s) =  \nu! h^\nu + O\(h^{\nu-1}\).$$
Take $s\in S$ and fix a nontrivial solution
$(x_1^*,\ldots,x_\nu^*, y_1^*,\ldots,y_\nu^*)$ of~\eqref{eq:asym cong x,y}.
Define the polynomial
$$R^*(Z)=(x_1^*+Z)\ldots(x_\nu^*+Z) - (y_1^*+Z)\ldots(y_\nu^*+Z).$$
Clearly, $R^*$ is not a zero polynomial and $R^*(s)\equiv0\pmod p$.
In particular, since $1/(\nu^2-1) < 1/\nu$ for $\nu\ge 2$ we
see that $R^*$ is not a constant polynomial,
assuming that $p$ is sufficiently large.

We decompose $R^*$ as a product of irreducible over $\Q$ polynomials
$$R^*(z)=W_1(Z)\ldots W_n(Z)$$
with $W_j\in\Z[Z]$ for $j=1,\ldots,n$. We have $W_j(s)\equiv0\pmod p$
for some $j$. Denote $W^*=W_j$.

Now we consider any nontrivial solution
$(x_1,\ldots,x_\nu, y_1,\ldots,y_\nu)$ of~\eqref{eq:asym cong x,y}
and define the polynomial
$$R(Z)=(x_1+Z)\ldots(x_\nu+Z) - (y_1+Z)\ldots(y_\nu+Z).$$
Again, $R$ is not a zero polynomial and $R(s)\equiv0\pmod p$.
As in the above we see that $R$ is not a constant polynomial.

Writing
$$
R(Z)= \sum_{j=0}^{\nu-1} r_j Z^{\nu-1 -j}
$$
we see that
\begin{equation}
\label{eq:Bound rj}
r_j \ll h^{j+1},\qquad j=0, \ldots,\nu-1 .
\end{equation}

Clearly $R^*$ satisfies the same bound. So writing
$$
W^*(Z)= \sum_{j=0}^{\mu-1} w_j Z^{\mu-1 -j}
$$
for some $\mu \le \nu$ and applying Lemma~\ref{lem:PolCoef},
we infer
\begin{equation}
\label{eq:Bound wj}
w_j \ll h^{j+1},\qquad j=0, \ldots,\mu-1.
\end{equation}

We have $\Res(R,W^*)\equiv0\pmod p$ since
$R(s)\equiv W^*(s)\equiv0\pmod p$.
Next, by Lemma~\ref{lem:Res}, (applied with $\rho=\vartheta=1$) we see that
$$\Res(R,W^*)  \ll h^{\nu^2-1}.
$$
Thus, provided that $\gamma(\nu)$ is small enough and $p$
is large enough we obtain the inequality
$$|\Res(R,W^*)|<p.
$$
 Hence, $\Res(R,W^*)=0$. Using irreducibility of the
polynomial $W^*$ we conclude that $R$ is divisible by $W^*$ in $\Q[Z]$.

Consider a mapping $\Phi:\, S\to\C$ by associating with any $s\in S$
a zero $\sigma$ of a corresponding polynomial $W^*$. We see from the above discussion that
any solution
$(x_1,\ldots,x_\nu, y_1,\ldots,y_\nu)$ of congruence~\eqref{eq:asym cong x,y}
induces the same solution of the equation~\eqref{eq:eq x,y}. Thus,
$$K_{\nu}(p,h,s) = K_{\nu}(h,\sigma).$$
Using Theorem~\ref{thm:RatProdSet s}, we get
$$
K_{\nu}(p,h,s) =  \nu! h^\nu + O\(h^{\nu-\varepsilon/2}
\exp\(O\(\frac{\log h}{\log\log h}\)\)\)
$$
unless $\sigma=\Phi(s)$ is an element of an exceptional subset of $\C$
of cardinality  $O(h^{1+\varepsilon})$. Taking into account that
a preimage $\Phi^{-1}(\sigma)$ contains at most
$n \le \nu-1$  elements for any
$\sigma\in\C$, we complete the proof.
\end{proof}

\section{Distribution of Elements of Large Multiplicative Order}
\label{sec:Ord}

\subsection{Distribution in very short   intervals
for almost all $p$}

Let $\ordp a$ denote  the multiplicative order of $a\in \F_p^*$.
Our aim is to prove that for almost all primes $p$ very short intervals in $\F_p$
(including intervals of  fixed size) contain an element of large multiplicative
order.

For $h=2$, Chang~\cite{Chang2} has shown that
for any function $\eta(z)>0$ with $\eta(z) \to 0$ as $z\to \infty$ and sufficiently large positive integer  $T$,
for all but $o(\pi(T))$ primes $p \le T$,  for all but $O(1)$ elements
$s \in \F_p$, we have
\begin{equation}
\label{eq:ConsecEll}
\max\{\ordp s, \ordp (s+1)\} >  p^{1/4+\eta(p)}.
\end{equation}
In fact her argument implies that~\eqref{eq:ConsecEll} hold
for almost all primes $p$ and
for any $s \in \F_p^*$ with $\ordp s > 3$.

We note that if $\ordp s=3$ then $s+1\equiv -s^2\pmod p$ thus $\ordp (s+1) = 6$.

For small $h$ we have the following result.

\begin{theorem}
\label{thm: small interval large order}
Let  $\eta(z)>0,\, m(z)>0$ be arbitrary functions with $\eta(z) \to 0$ and $m(z)\to\infty$ as $z\to \infty$. Then
for all but $o(\pi(T))$ primes $p \le T$, any interval $I=[t,t+m(T)]$ has an element $\xi$ for which
$$
\ordp \xi>T^{1/2+\eta(T)}.
$$
\end{theorem}

\begin{proof} We assume that $T$ is large enough.
Using Lemma~\ref{lem:MultIndSubset beta+j} we take $c=c(1/2)>0$.
Note that either by increasing $\eta(T)$ or decreasing $m(T)$,
we may assume that~$\eta(z)>(\log z)^{-1/4}$ and
that
\begin{equation}
\label{eqn:ErdMurDefine m}
m(T)=\fl{\eta(T)^{-C/\eta(T)}},\quad C=3/c.
\end{equation}
holds. Next, denote
\begin{equation}
\label{eqn:ErdMurDefine rNU}
m = m(T), \quad r=\fl{c\log m}, \quad \mu= \fl{m^{-1}T^{1/2(r+2)}}.
\end{equation}
Given $\cE_1,\cE_2\subseteq \{1,\ldots,m\}$,  $\cE_1\cap \cE_2=\emptyset$ and
$\# \cE_1+ \# \cE_2\le r$, $\widetilde{\mu}=(\mu_j)_{j\in \cE_1\cup \cE_2}$ with
$1\le \mu_j\le \mu$, denote
$$
P_{\cE_1,\cE_2,\widetilde{\mu}}(X)=\prod_{j\in \cE_1}
(X+j)^{\mu_j}-\prod_{j\in \cE_2}(X+j)^{\mu_j}\in\Z[X].
$$
These polynomials are of degree at most $r\mu$ and logarithmic height
$O(r\mu \log m)$. Their number is clearly bounded by
$2^{r+1}\binom{m}{r}\mu^r$ (here we have used that $r\le m/2$).

Factor each polynomial $P_{\cE_1,\cE_2,\widetilde{\mu}}(X)$  in irreducible
(over $\Q$) factors $f\in \Z[X]$ and let $\Pf\subseteq \Z[X]$ be
the set of all polynomials
obtained this way. Hence,
\begin{equation}
\label{eqn:ErdMurCardinalityPi}
\# \Pf\le r2^{r+1}\binom{m}{r}\mu^{r+1}.
\end{equation}
Denote then
$$
R=\prod_{\substack{f,g\in\Pf\\ \Res(f,g)\not=0}}\Res(f,g)\in\Z.
$$
Using Corollary~\ref{cor:twoPol} we see that all coefficients of
any polynomial $f\in\Pf$ are bounded by $(O(m))^{r\mu}$. Next,
using a straightforward bound on the resultants $\Res(f,g)$,
by~\eqref{eqn:ErdMurDefine rNU}
and~\eqref{eqn:ErdMurCardinalityPi}, we derive
$$
|R|\le m^{O(r^4 4^r{\binom{m}{r}}^2\mu^{2(r+2)})}<2^{o(T)}.
$$
Let $\cP_1$ and $\cP_2$ be the sets of exceptional primes in
 Lemma~\ref{lem:ErdMur1} and Corollary~\ref{cor:ErdMur2} respectively.
Hence, denoting
\begin{equation*}
\begin{split}
\cP_0 & = \{p<T~:~ \gcd(p,R)=1,\   p \not\in \cP_1 \cup \cP_2 \},
\end{split}
\end{equation*}
we have
$$
\# \cP_0=(1+o(1))\pi(T).
$$

Take $p\in \cP_0$.
Since $p\not \in \cP_1$,
the desired statement is equivalent to
\begin{equation}
\label{eqn: smaller interval large order}
\ordp \xi\ge T^{1/2-\eta(T)}.
\end{equation}
Let  $I=[t,t+m(T)]\subseteq \F_p$ and assume that
\begin{equation}
\label{eqn: smaller interval large order assume contrary}
\ordp \xi<T^{1/2-\eta(T)} \quad {\rm for\,\, all} \quad \xi\in I.
\end{equation}
Since $p\not \in \cP_2$,  we may write $p-1=AB$, where
$A\le T^{\eta(T)/2}$ and $B$ has no prime factors less than
$T^{c\eta(T)/(\log(1/\eta(T)))}$.
Hence, $B$ has at most $(1/\eta(T))^{1/(c\eta(T))}$ divisors.
For $\xi\in I$ we write
$$
\frac{p-1}{\ordp \xi}=\frac{AB}{\ordp \xi}=ab
$$
where $a\mid A,\, b\mid B$. In particular, $\xi=g^{ab}$ for some primitive root $g$ modulo $p$. Thus, $\xi$ belongs to the subgroup $\cH$ of $\F_p^*$
generated by the element $g^b$.

Since $A\le T^{\eta(T)/2}$, using~\eqref{eqn: smaller interval large order assume contrary}, we derive that
$$
b=\frac{p-1}{a \ordp \xi}> \frac{p-1}{T^{(1-\eta(T))/2}}.
$$
Thus,
$$
\# \cH\le\frac{p-1}{b}<T^{(1-\eta(T))/2}.
$$
Since $B$ has at most $(1/\eta(T))^{1/(c\eta(T))}$ divisors, the above argument shows that we have a family of at most $(1/\eta(T))^{1/(c\eta(T))}$ subgroups $\cH$ of $\F_p^*$ of size
$$
\# \cH <T^{(1-\eta(T))/2}
$$
and such that each element $\xi\in I$ is contained in one of these groups.
Consequently,
we conclude that there is a  subgroup $\cH$ of $\F_p^*$  of order
\begin{equation}
\label{eqn: fixed H=N upper bound}
\#\cH =N<T^{(1-\eta(T))/2}
\end{equation}
such that for the set  $\cS\subseteq [1, m]$, defined by
$$
t+\cS = \cH \cap I,
$$
by the choice of parameters $c$, $m$ and $r$
in~\eqref{eqn:ErdMurDefine m} and~\eqref{eqn:ErdMurDefine rNU},  we have
$$
\# \cS>\sqrt{m}.
$$
Let $\cE\subseteq \cS$, $\#\cE=r$. Assuming all elements
$$
\prod_{j\in \cE}(t+j)^{\mu_j}\in \cH
$$
with $0\le \mu_j\le \mu$ are distinct modulo $p$, it follows from~\eqref{eqn:ErdMurDefine m}
and~\eqref{eqn:ErdMurDefine rNU} that
$$
N\ge \mu^r>T^{(1-\eta(T))/2}
$$
contradicting~\eqref{eqn: fixed H=N upper bound}.

Hence, for each $\cE\subseteq \cS$ with $\# \cE=r$, there are
disjoint $\cE_1,\cE_2\subseteq \cE$ and exponents  $\widetilde{\mu}=(\mu_j)_{j\in \cE_1\cup \cE_2}$ such that
$$
P_{\cE_1,\cE_2,\widetilde{\mu}}(t)\equiv 0\pmod p.
$$
We may then extract an irreducible (over $\Q$) factor $f_E\in \Pf$ from
$P_{\cE_1,\cE_2,\widetilde{\mu}}$ such that
$$
f_E(t)\equiv 0\pmod p.
$$
Thus, for all $\cE,\cF\subseteq S$ with $\# \cE=\# \cF=r$
$$
\Res (f_\cE,f_{\cF})\equiv 0\pmod p,
$$
while also $\Res (f_\cE,f_{\cF})\mid R$. Since $\gcd(p,R)=1$, it follows that necessarily
$$
\Res (f_\cE,f_{\cF})=0
$$
for all $\cE,\cF \subseteq \cS$ of size $\# \cE=\# \cF=r$. Since the $f_\cE$ are irreducible,
they must coincide
up to a scalar factor
and hence have a common root $\beta\in \C$.

Apply then Lemma~\ref{lem:MultIndSubset beta+j} to $\cA=\cS+\beta\subseteq \{\beta+j; \, 1\le j\le m\}\subseteq \C$. This gives a multiplicatively independent
set $\cA_0=\cE+\beta$ with some $\cE\subseteq \cS$  of size $\# \cE=r$. But since $f_\cE(\beta)=0,$ we get
$P_{\cE_1,\cE_2,\widetilde{\mu}}(\beta)=0$, contradicting the multiplicative independence.
\end{proof}

\subsection{Distribution in very short intervals
for a large proportion of primes $p$}

\begin{theorem}
\label{thm: compare Erdos Murty}
If $a\in \N$, $m>1$, $1\le \mu<\log m$  and $T\in\Z_+$ is taken
sufficiently large, then for all but  $O(\mu^{-1} \pi(T) )$ primes $p<T$ we have
$$
\max_{0\le j<m}\ordp(a+j)>T^{1-m^{-1/\mu}}.
$$
\end{theorem}

\begin{proof}  For small $\mu$ (and, in particular, for small $m$)
the result is trivial. We  assume that $\mu$ is large enough.
Moreover, it is enough to prove the result for
\begin{equation}
\label{assum_for_mu}
\mu<0.1\log m.
\end{equation}
Take $r=4\fl{m^{1/\mu}}$. Then we have
\begin{equation}
\label{est_for_r}
r\le m^{0.1}.
\end{equation}
Let $\mathfrak{L}$ be the collection of all multiplicatively independent
subsets $\cS\subseteq I=[a, a+m)$ of cardinality $r$.
The set $\mathfrak{L}$ is nonempty by
Corollary~\ref{mult_ind}. Let $T$ be sufficiently large (depending on $a,m$)
and set
\begin{equation*}
\begin{split}
K&=\fl{T^{1/(r+2)}},\\
R&=\prod_{\substack{\cS\in \mathfrak{L}\\{\cS_1,\cS_2\subseteq \cS, \cS_1\cap \cS_2=\emptyset}}}
\(\prod_{\substack{1\le k_\xi\le K \\ \xi\in \cS_1}}\xi^{k_\xi}-
\prod_{\substack{1\le k_\xi\le K \\ \xi\in \cS_2}}\xi^{k_\xi}\).
\end{split}
\end{equation*}
Hence, $R\in \Z\setminus{\{0\}}$ and for $T$ large enough
$$
|R|<(a+m)^{rK^{r+1}3^r}<2^{o(T)}.
$$
Hence, denoting
$$\cP_0 = \{p<T~:~\gcd(p,R)=1\}$$
we have
\begin{equation}
\label{eqn:est_ P0}
\#\cP_0=(1+o(1))\pi(T).
\end{equation}

Now we take
$$\alpha=m^{-1/\mu}/\mu,\quad\beta=1/\mu, \quad \gamma=m^{-1/\mu}/2$$
and denote by $\cP_1$ the set of primes $p\in\cP_0$ such that
\begin{itemize}
\item[(i)] the product of all prime factors of $p-1$ that are smaller than
$T^{\alpha}$ is at most $T^{\gamma}$;
\item[(ii)] the number of prime divisors $q$ of $p-1$
satisfying $q\ge T^\alpha$ does not exceed
$$\frac{0.2\mu\log(\beta/\alpha)}{1-\beta}+\frac{1}{\beta};$$
\item[(iii)] $p-1$ has no divisor $q^2$ with $q\ge T^\alpha$.
\end{itemize}

By~\eqref{eqn:est_ P0} and Lemmas~\ref{lem:ErdMur2}, \ref{lem:NumDiv},
and~\ref{lem:squaredivides}, we derive
$$\#\cP_1= (1 + O(1/\mu)) \pi(T).
$$
Also, observe that, by~\eqref{assum_for_mu}, for large enough $\mu$ we have
\begin{equation}
\label{eqn:prime_fact_B}
\frac{0.2\mu\log(\beta/\alpha)}{1-\beta}+\frac{1}{\beta}<0.3\log m+\mu <0.4\log m.
\end{equation}
By assumption (i) in the definition of $\cP_1$, we may write $p-1=AB$,
where $A<T^{\gamma}$ and $B$ has no prime factors less than $T^{\alpha}$.
By (ii) and~\eqref{eqn:prime_fact_B}, $B$ has at most $0.4\log m$ prime
factors. By (iii), $B$ is square-free. Therefore, the number of factors of $B$
is at most
$$2^{0.4\log m}<m^{0.3}.$$
We assume that for $j=0,\ldots,m-1$ we have
$$\ordp(a+j)\le T^{1-m^{-1/\mu}}.$$
By Corollary~\ref{mult_ind}, we choose a set
$$\cI\subseteq\{a+j~:~j\in[0,m[\},\qquad \#\cI\gg\sqrt m/\log m,$$
of multiplicatively independent numbers.
For $\xi\in \cI$ we write
$$
\frac{p-1}{\ordp\xi}=\frac{AB}{\ordp \xi}=ab
$$
where $a\mid A,\, b\mid B$.
Since $B$ has at most $m^{0.3}$ divisors, we can take a subset $\cJ\subseteq \cI$
with $\#\cJ\ge\#\cI m^{-0.3}> m^{0.1}$ such that the same divisor $b$
of $B$ is associated to any $\xi\in \cJ$. We have $\xi=g^{ab}$ for some primitive
root $g$ modulo $p$. Thus, $\xi$ belongs to the subgroup $\cH$ of $\F_p^*$
generated by the element $g^b$. We have
$$
b=\frac{p-1}{a \ordp \xi}> \frac{p-1}{T^{\gamma}\ordp \xi}
\ge\frac{p-1}{T^{1-0.5m^{-1/\mu}}}.
$$
Thus, $\cJ\subseteq \cH $ and
$$
\#\cH \le\frac{p-1}{b}<T^{1-0.5m^{-1/\mu}}.
$$
By~\eqref{est_for_r}, we can take a subset $\cS\subseteq \cJ$ of cardinality $\# \cS =r$.
Since $p\in\cP_0$, we conclude that all products
$$\prod_{\substack{0\le k_\xi\le K \\ \xi\in \cS}}\xi^{k_\xi}$$
are distinct modulo $p$.
The number of such products is
$$(K+1)^r>T^{r/(r+2)}>T^{1-0.5m^{-1/\mu}}>\# \cH.$$
But this is impossible since all the products belong to $\cH$.
This completes the proof.
\end{proof}

\subsection{Distribution in longer  intervals
for almost all $p$}
For large $h$ we can use Theorem~\ref{thm:GProdSet AlmostAll}.

\begin{theorem}
\label{thm:Order Fp}  Let $\alpha>0$ be fixed.
For  $T^\alpha\le h<T$,
for all but $o(\pi(T))$ primes $p \le T$ and for any  $s \in\F_p$, the set
$$\cA = \{ x+s  ~:~ 1\le x\le h \} \subseteq \F_p
$$
contains an
element $a \in \cA$ of multiplicative
order
$$
\ordp a > \exp\(O\(\frac{\log h}{\log\log h}\)\)(hT)^{1/2}.
$$
\end{theorem}

\begin{proof} We fix $\nu\ge1/\alpha$.
Clearly $\cA$ contains a set $\cB \subseteq \cA$  of elements of the same
multiplicative order $t$ and of cardinality
\begin{equation}
\label{eq:Large B} \#\cB \ge \# \cA/ \tau(p-1) = h/\tau(p-1).
\end{equation}
For $\lambda \in \F_p^*$, let
$$
Q(\lambda) = \#\{(x_1,\ldots, x_\nu)\in\cB\times\ldots \times \cB~:~
\lambda \equiv x_1\ldots x_\nu \pmod p\}.
$$
Then obviously,
$$
\#\{\lambda \in \F_p ~:~ Q(\lambda) > 0\} \le t.
$$
Hence, using the Cauchy inequality, we obtain
$$
(\#\cB)^{2\nu} = \(\sum_{\lambda \in \F_p^*}Q(\lambda)\)^2
\le t \sum_{\lambda \in \F_p^*}Q(\lambda)^2 \le t K_\nu(p,h,s),
$$
which together with Theorem~\ref{thm:GProdSet AlmostAll}
and the standard estimate for
$\tau(p-1)$ implies the result.
\end{proof}

We note that for intermediate values of $h$, namely for $h$ with
$T^\alpha\le h<T^{1-\alpha}$ for some fixed $\alpha>0$, using the
ideas and results of  Erd\H os and Murty~\cite{ErdMur} one can improve
slightly Theorem~\ref{thm:Order Fp}. Namely, one can show that
for any function $\eta(z)>0$ with $\eta(z)\to 0$ we have
$$
\ordp a >  (hT)^{1/2}T^{\eta(T)}
$$
instead of the bound of Theorem~\ref{thm:Order Fp}.

\section*{Acknowledgement}

The research of  J.~B. was partially supported
by National Science Foundation   Grant DMS-0808042, that of  S.~V.~K.
by Russian Fund for Basic Research, Grant N.~11-01-00329,
and Program Supporting Leading Scientific Schools, Grant Nsh-6003.2012.1,
and that of  I.~E.~S. by Australian Research Council Grant  DP1092835.

\end{document}